\documentclass[a4paper,10pt]{amsart}

\usepackage[english]{babel}
\usepackage[utf8]{inputenc}
\usepackage[T1]{fontenc}
\usepackage{amsmath}
\usepackage{amsfonts}
\usepackage{amssymb}
\usepackage{amsthm}
\usepackage{lmodern}

\usepackage{hyperref}

\newtheorem{theo}{Theorem}[subsection]
\newtheorem*{theo*}{Theorem}
\newtheorem{prop}[theo]{Proposition}
\newtheorem{lemm}[theo]{Lemma}
\newtheorem{coro}[theo]{Corollary}
\newtheorem{conj}[theo]{Conjecture}

\theoremstyle{remark}
\newtheorem{rema}[theo]{Remark}

\newcommand{\C}{\mathbb{C}}
\newcommand{\Z}{\mathbb{Z}}
\newcommand{\F}{\mathbb{F}}
\newcommand{\Q}{\mathbb{Q}}
\renewcommand{\O}{\mathcal{O}}
\newcommand{\eps}{\varepsilon}

\newcommand{\spec}{\operatorname{Spec}}

\newcommand{\Hom}{\operatorname{Hom}}
\newcommand{\ind}{\operatorname{ind}}

\newcommand{\Frob}{\operatorname{Frob}}
\newcommand{\Fil}{\operatorname{Fil}}

\renewcommand{\bar}[1]{\overline{#1}}

\newcommand{\GL}{\operatorname{GL}}

\newcommand{\Sym}{\operatorname{Sym}}
\newcommand{\Gal}{\operatorname{Gal}}

\newcommand{\mat}[4]{%
\begin{pmatrix} #1 & #2\\
#3 & #4 
\end{pmatrix}%
}
\newcommand{\smat}[4]{%
\left(\begin{smallmatrix} #1 & #2\\
#3 & #4 
\end{smallmatrix}\right)%
}

\newcommand{\T}{\mathbb{T}}
\newcommand{\f}{\bar{\F}}
\newcommand{\z}{\bar{\Z}_p}

\newcommand{\A}{\mathbb{A}}

\renewcommand{\aa}{\mathcal{A}}

\newcommand{\m}{\mathfrak{m}}

\renewcommand{\t}{\texttt{t}}
\newcommand{\G}{\mathcal{G}}
\newcommand{\R}{\mathcal{R}}

\newcommand{\JL}{\operatorname{JL}}
\newcommand{\LL}{\operatorname{LL}}
\newcommand{\WD}{\operatorname{WD}}
\newcommand{\Norm}{\operatorname{N}_D}
\newcommand{\WG}{W_{\Gamma}}
\newcommand{\triv}{\operatorname{triv}}

\newcommand{\scal}{(\text{scal})}
\newcommand{\red}{(\text{red})}
\newcommand{\irr}{(\text{irr})}
\newcommand{\car}{(\text{char})}

\newcommand{\unr}{\operatorname{unr}}

\newcommand{\Art}{\text{Art}}

\newcommand{\sigmaG}{\sigma_{\G}}
\newcommand{\rec}{\text{rec}_p}

\title[Discrete series extended types]{Potentially semi-stable deformation rings for discrete series extended types}

\date{\today}
\author{Sandra Rozensztajn}
\address{UMPA, \'ENS de Lyon\\
UMR 5669 du CNRS\\
46, all\'ee d'Italie\\
69364 Lyon Cedex 07\\
France}
\email{sandra.rozensztajn@ens-lyon.fr}

\subjclass{11F80,11F33}

\begin{document}

\maketitle

\begin{abstract}
We define deformation rings for potentially semi-stable deformations of
fixed discrete series extended type in dimension $2$.
In the case of representations of the Galois group of $\Q_p$, we prove an
analogue of the Breuil-Mézard conjecture for these rings. As an
application, we give some results on the existence of congruences modulo
$p$ for newforms in $S_k(\Gamma_0(p))$.  
\end{abstract}

\section{Introduction}
\label{intro}

Let $p > 2$ be a prime number, $K$ a finite extension of $\Q_p$ with
absolute Galois group $G_K$. Let $\bar{\rho}$ be a continuous
representation of $G_K$ of dimension $2$ with coefficients in some finite
field $\F$ of characteristic $p$. 
Let $E$ be a finite extension of $\Q_p$
with residue field containing $\F$. There exists an $\O_E$-algebra
$R^{\square}(\bar{\rho})$ parametrizing the framed deformations of
$\bar{\rho}$ to $\O_E$-algebras. Kisin (\cite{Kis08}) has shown that this
ring has quotients $R^{\square,\psi}(w,\t,\bar{\rho})$ that parametrize
framed deformations $\rho$ that are potentially semi-stable of given determinant
(encoded by $\psi$), fixed Hodge-Tate weights (encoded by the Hodge-Tate
type $w$) and fixed inertial type $\t$ (that is,
the restriction to inertia of the Weil-Deligne representation
$\WD(\rho)$ associated to $\rho$ is isomorphic to a fixed smooth
representation $\t$). We are interested in a variant of this situation:
instead of considering deformations with a fixed inertial type $\t$,
consider deformations with a fixed extended type $\t'$, that is
such that the restriction to the Weil group of $\WD(\rho)$ is isomorphic
to $\t'$, in the case when $\t'$ is a discrete series type (see the
definition in Paragraph \ref{exttypesWD}). This problem was first
considered in \cite{BCDT}, in order to isolate some irreducible
components of the deformation space parametrizing deformations with fixed inertial type.  
For a discrete series inertial type $\t$, we show that the ring
$R^{\square,\psi}(w,\t',\bar{\rho})$ parametrizing deformations with
fixed discrete series extended type $\t'$ extending $\t$ is the maximal
reduced quotient of $R^{\square,\psi}(w,\t,\bar{\rho})$ supported in some
set of irreducible components of $\spec
R^{\square,\psi}(w,\t,\bar{\rho})$.  More precisely, depending on $\t$,
adding the extra data of a $\t'$ either does not give any additional
information, or divides the set of irreducible components in two parts.

Some important information about the geometry of the rings 
$R^{\square,\psi}(w,\t,\bar{\rho})$ is given by the Breuil-Mézard
conjecture (\cite{BM}, proved for $K=\Q_p$ by Kisin \cite{Kis09a} and
Pa\v{s}k\=unas \cite{Pas}) that relates the Hilbert-Samuel multiplicity
of the special fiber of the ring to an automorphic multiplicity, computed
in terms of smooth representations modulo $p$ of $\GL_2(\O_K)$ attached to $w$ and $\t$.
Our main result is that when $K=\Q_p$ there is a
similar formula for the Hilbert-Samuel multiplicity of the special fiber
of $R^{\square,\psi}(w,\t',\bar{\rho})$ for a discrete series extended
type $\t'$. More precisely, Gee and Geraghty have shown in \cite{GG} that
for discrete series inertial types $\t$, the Breuil-Mézard conjecture can
be reformulated using an automorphic multiplicity expressed in terms of
representations not of $\GL_2(\O_K)$, but of $\O_D^{\times}$, where
$\O_D$ is the ring of integers of the non-split quaternion algebra $D$
over $K$. The formula we give for the multiplicity of the special fiber
of $R^{\square,\psi}(w,\t',\bar{\rho})$ is in terms of representations of
a quotient $\G$ of $D^{\times}$ containing $\O_D^{\times}$ as a subgroup
of index $2$. Using the local Langlands correspondence and the
Jacquet-Langlands correspondence, we construct for each discrete series
inertial type $\t$ a smooth representation $\sigmaG(\t)$ of $\G$ with coefficients in
$\bar{\Q}_p$ (or a pair of such representations, depending on the
inertial type $\t$). To a Hodge-Tate type $w$ we attach a representation
$\sigma_w$ of $\G$ coming from an algebraic representation of $\GL_2$
with highest weight given by $w$. The Hilbert-Samuel multiplicity is then
given in terms of the multiplicity of the irreducible constituents of the
reduction modulo $p$ of $\sigmaG(\t)\otimes\sigma_w$, seen as
representations of a finite group $\Gamma$ through which all semi-simple
representations modulo $p$ of $\G$ factor.
For $K=\Q_p$ we have the following theorem (see Theorem \ref{mainQp} for
a more precise statement):

\begin{theo*}
Let $\bar{\rho}$ be a continuous representation of $G_{\Q_p}$ of dimension $2$
with coefficients in $\f_p$.
There exists a positive linear form $\mu_{\bar{\rho}}$ on the
Grothendieck ring of representations of $\Gamma$
with values in $\Z$ satisfying the following property:
for any discrete series inertial type $\t$,
Hodge-Tate type $w$, character $\psi$ lifting
$\omega^{-1}\det\bar{\rho}$, and extended type
$\t'$ compatible with $(\t,\psi)$,
there exists a choice of representation $\sigmaG(\t)$ of $\G$
such that we have:
\[
e(R^{\square,\psi}(w,\t',\bar{\rho})/\pi) =
\mu_{\bar{\rho}}([\bar{\sigmaG}(\t)\otimes\bar{\sigma_w}])
\]
\end{theo*}

We deduce our result from the reformulation by \cite{GG} of
the usual Breuil-Mézard conjecture, making use of modularity lifting
theorems for modular forms on a quaternion algebra ramified at infinity
and at primes dividing $p$.

One consequence of this formula is Corollary \ref{samemult}: except when
$\bar\rho$ has some very specific form, then when the addition of the
data of the extended types divides the deformation ring in two parts,
these parts have the same multiplicity.
This is to be
expected when $\bar\rho$ is irreducible, as in this case it can easily
be seen that the deformation rings corresponding to the two extended
types are in fact isomorphic (see Remark \ref{involution}). But this is
much more surprising when $\bar\rho$ is reducible, as in this case there
does not seem to be a natural way to relate the deformation rings corresponding
to the two extended types.

We give a concrete application of our result to the existence of
congruences modulo $p$ for some modular forms. When $\t$ is trivial, the ring
$R^{\square,\psi}(w,\t,\bar{\rho})$ classifies semi-stable
representations, and the extra data given by the extended type is the
eigenvalues of the Frobenius of the associated filtered $(\phi,N)$-module
when the representation is not crystalline
(there are only two possibilities for these eigenvalues if the
determinant is fixed). 
If $f \in S_k(\Gamma_0(p))$ is a newform, this
means that the extended type of ${\rho_{f,p}}|_{G_{\Q_p}}$ gives the value of
the coefficient $a_p(f) = \pm p^{k/2-1}$. We give in Theorem \ref{congruences} a
criterion for the existence of a newform in $S_k(\Gamma_0(p))$ that is
congruent to $f$ modulo $p$ but with the opposite value for $a_p$.

I would like to thank Christophe Breuil and Vincent Pilloni for useful
conversations. I would also like to thank the referees for their useful
comments and particularly for suggesting a better proof of Proposition
\ref{defextended}.

\subsection{Plan of the article}

We define the deformation rings $R^{\square,\psi}(w,\t',\bar{\rho})$ for
discrete series extended types in Section \ref{galois}. In Section
\ref{automorphic} we introduce the groups and representations that play a
role in the automorphic side for the formula for the Hilbert-Samuel
multiplicity of the special fiber of the rings
$R^{\square,\psi}(w,\t',\bar{\rho})$ and state our main theorems. We give
in Section \ref{quatmodforms} some results about modular forms for
quaternion algebras ramified at infinity and at primes dividing $p$ that
we need in Section \ref{proof}, where we prove the theorems. Section
\ref{appli} is devoted to the application to modular forms.

\subsection{Notation}
\label{notation}

We fix a prime number $p > 2$.
We denote by $K$ a finite extension of $\Q_p$, and by $q$ the cardinality
of its residue field. Let $G_K$ be the absolute Galois group of $K$,
$I_K$ its inertia subgroup and $W_K$ its Weil group. 
We denote by $\eps$ the $p$-adic cyclotomic
character and $\omega$ its reduction modulo $p$.
We normalize the Artin map of local class field theory $\Art_K : K^{\times} \to
W_K^{ab}$ so that geometric Frobenius elements correspond to
uniformizers.
We denote by $\unr(a)$ the unramified character of $W_K$ (or $G_K$)
sending a geometric Frobenius to $a$, and also the unramified character
of $K^{\times}$ sending a uniformizer to $a$. We denote by $\|\cdot\|$ the
norm on $W_K$, that is the character $\unr(q^{-1})$.
\section{Discrete series extended types and deformation rings}
\label{galois}

\subsection{Extended types and Weil-Deligne representations}
\label{exttypesWD}

An inertial type is a smooth representation $\t$ of $I_K$ over $\bar\Q_p$ 
that extends to
a representation of $W_K$. We define an extended type to be a smooth
representation of $W_K$ over $\bar\Q_p$.

We recall the following well-known classification for inertial types and
extended types in dimension $2$ when $p>2$ (see for example
\cite[Lemma 2.1]{Imai} for a proof):
\begin{lemm}
\label{classtypes}
Let $\t$ be an extended type of degree $2$. Then we are in exactly one of
the following situations:
\begin{description}
\item[$\scal$] 
$\t|_{I_K}$ is scalar: there exist two smooth characters $\chi,\chi'$ of $W_K$ such that
$\chi|_{I_K} = \chi'|_{I_K}$ and $\t = \chi \oplus \chi'$.

\item[$\car$]
There exist two smooth characters $\chi_1$, $\chi_2$ of $W_K$ with distinct restrictions to 
$I_K$ such that $\t = {\chi_1}\oplus {\chi_2}$.

\item[$\red$]
Let $K'$ be the unramified quadratic extension of $K$.
There exists a smooth character $\chi$ of $W_{K'}$
that does not extend to a character of $W_K$ such that
such that $\t = \ind_{W_{K'}}^{W_K}\chi$. In this case $\t|_{I_K}$ is
reducible and is a sum of characters that do not extend to $W_K$.

\item[$\irr$]
There exist a ramified quadratic extension $L$ of $K$ and a smooth character
$\chi$ of $W_L$ that does not extend to a character of $W_K$ such that
$\t = \ind_{W_L}^{W_K}\chi$. In this case $\t|_{I_K}$ is irreducible.
\end{description}
\end{lemm}

We call the inertial types corresponding to situation $\scal$, $\red$ or $\irr$
discrete series inertial types. We call the extended types corresponding
to situation $\red$ or $\irr$, or to situation $\scal$ with $\chi' =
\chi\otimes\|\cdot\|^{\pm 1}$ discrete series extended types.

The following Proposition is an immediate consequence of the
classification:

\begin{prop}
\label{two}
Let $\t_1$ and $\t_2$ be two discrete series extended types with isomorphic
restrictions to $I_K$. Then they differ by a twist by an unramified
character.

Let $\t$ be a discrete series extended type. If it 
is of the form $\scal$ or $\irr$ then $\t$ is not isomorphic to
$\t\otimes\unr(-1)$. If $\t$ is of the form $\red$ then $\t$ is
isomorphic to $\t\otimes\unr(-1)$.
\end{prop}

Let $\t$ be a discrete series extended type. We call conjugate type of $\t$ the type
$\t\otimes\unr(-1)$. Two types with isomorphic restriction to
$I_K$ are conjugate if and only if they have the same determinant.
When $\t$ is of the form $\scal$ or $\irr$, two conjugate extended types
are distinct, but they are isomorphic when $\t$ is of the form $\red$.

Let $(r,N)$ be a Weil-Deligne representation of dimension $2$, that is a
two-dimensional smooth representation $r$ of the Weil group $W_K$ and a
nilpotent endomorphism $N$ such that $Nr(x) = \|x\|^{-1}r(x)N$ for any
$x\in W_K$.
Let $\t$ be an inertial type; we say that $(r,N)$ is of inertial type
$\t$ if $r|_{I_K}$ is isomorphic to $\t$. Let $\t'$ be an extended type;
we say that $(r,N)$ is of extended type $\t'$ if $r$ is isomorphic to
$\t'$.

We say that $(r,N)$ is a discrete series Weil-Deligne representation if
either $r|_{I_K}$ is of the form $\scal$ and $N \neq 0$ or $r|_{I_K}$ is
of the form $\red$ or $\irr$ (note that we can have $N \neq 0$ only when
$r|_{I_K}$ is of the form $\scal$ and $r$ is a twist of
$1\oplus\|\cdot\|$). With this definition, discrete series inertial
(resp. extended) types are exactly the restriction to $I_K$ (resp. $W_K$)
of discrete series Weil-Deligne representations. See Paragraph
\ref{lljl} for a justification of this terminology.

\subsection{Potentially semi-stable representations and discrete series extended
types}

\subsubsection{Filtered $(\phi,N)$-modules with descent data}

Let $F$ be a finite extension of $\Q_p$, $F_0$ the maximal unramified
extension of $\Q_p$ contained in $F$. Let $E$ be a finite extension of
$\Q_p$ (the coefficient field), that we suppose large
enough.

A filtered $(\phi,N,F,E)$-module is a free $F_0\otimes_{\Q_p}E$-module
$D$ of finite rank, endowed with a $F_0$-semi-linear, $E$-linear 
endomorphism $\phi$ and a $F_0\otimes E$ linear endomorphism $N$ 
satisfying the commutation relation 
$N\phi = p\phi N$, with $N$ nilpotent, $\phi$ an automorphism, and a
decreasing filtration of $F\otimes_{F_0}D$ by $F\otimes_{\Q_p}E$-submodules
$\Fil^i(F\otimes_{F_0}D)$ such that $\Fil^i(F\otimes_{F_0}D) =
F\otimes_{F_0}D$ when $i$ is small enough and $\Fil^i(F\otimes_{F_0}D) =
0$ when $i$ is large enough.
We can define an admissibility condition for filtered
$(\phi,N,F,E)$-modules, we refer to \cite{Fonb} for the definition.

Let $\rho : G_F \to \GL(V)$ be a continuous representation, where $V$ is a
finite-dimensional $E$-vector space. If $\rho$ is semi-stable, we can
attach to it an admissible filtered $(\phi,N,F,E)$-module by taking
$D_{st}(V) = (B_{st}\otimes_{\Q_p}V)^{G_F}$. The functor
$V \mapsto D_{st}(V)$ gives an equivalence of
categories between the category of semi-stable representations of $G_F$
and the category of admissible filtered $(\phi,N,F,E)$-modules which
preserves dimension, and the Hodge-Tate weights of $\rho$ are the indices
$i$ with $\Fil^{-i}(F\otimes_{F_0}D) \neq \Fil^{-i+1}(F\otimes_{F_0}D)$
(so that $\eps$ has its Hodge-Tate weights equal to $1$).

Suppose now that we have $\rho : G_K \to \GL(V)$ a continuous
representation such that $\rho$ becomes semi-stable on a finite Galois
extension $F$ of $K$. Then we can attach to it an admissible filtered
$(\phi,N,F/K,E)$-module, that is an admissible filtered
$(\phi,N,F,E)$-module with descent data given by an action of $\Gal(F/K)$ which is
$F_0$-semi-linear and $E$-linear and commutes with $\phi$ and $N$.
The filtered
$(\phi,N,F,E)$-module is $D^F_{st}(V)$, that is $D_{st}(V|_{G_F})$.
This gives an equivalence of categories between the category of
representations of $G_K$ that become semi-stable over $F$ and the
category of admissible filtered $(\phi,N,F/K,E)$-modules.

\subsubsection{Weil-Deligne representation attached to a Galois
representation}

Let $\rho : G_K \to \GL(V)$ be a continuous representation of $G_K$,
where $V$ is a finite-dimensional vector space over a finite extension
$E$ of $\Q_p$. If $\rho$ is potentially semi-stable, 
we attach to its filtered $(\phi,N,F/K,E)$-module
a Weil-Deligne representation $WD(\rho)$ as in \cite{Fona} (see also
Appendix B. of \cite{CDT} for more detailed explanations of the
construction and its properties). It does not depend on the field $F$ over which
$K$ becomes semi-stable, and moreover it does not depend on the filtration
but only on $\phi$, $N$ and the action of $\Gal(F/K)$.

We say that $\rho$ is of inertial type $\t$ if $\WD(\rho)$ is of type
$\t$, and of extended type $\t'$ if $\WD(\rho)$ is. Note that
$\WD(\rho)$ is of scalar inertial type if and only if $\rho$ is
semi-stable up to twist, and in this case $N\neq 0$ if and only if $\rho$ is
semi-stable but not crystalline up to twist.

\subsection{Deformation rings}
\label{defring}

In this section we fix a discrete series inertial type $\t$. Note that
the notions of this Paragraph will be interesting only when $\t$ is of the form $\scal$ or
$\irr$ as we explain later.

Let $\bar{\rho}$ be a continuous representation of $G_K$ of dimension $2$
over a finite field $\F$ of characteristic $p$. Let $E$ be a finite
extension of $\Q_p$ with residue field containing $\F$.
We denote by $R^{\square}(\bar{\rho})$ the universal framed deformation
$\O_E$-algebra of $\bar{\rho}$.

\subsubsection{Deformation rings of fixed inertial type}

Let $w = (n_{\tau},m_{\tau}) \in (\Z_{\geq 0} \times
\Z)^{\Hom(K,\bar{\Q}_p)}$ be a Hodge-Tate type, $\t$ be an
inertial type, and $\psi$ a character of $G_K$. We are interested in
lifts $\rho$ of $\bar{\rho}$ that are potentially semi-stable, with
determinant $\psi\eps$, Hodge-Tate weights
$(m_{\tau},m_{\tau}+n_{\tau}+1)_{\tau}$ (we then say that $\rho$ has
Hodge-Tate type $w$), and inertial type $\t$.

In \cite{Kis08}, Kisin
shows that, after possibly enlarging $E$, there exists
a quotient $R^{\square,\psi}(w,\t,\bar{\rho})$ of $R^{\square}(\bar{\rho})$ that
has the following properties:
\begin{theo}
\label{definertial}
\begin{enumerate}
\item
$R^{\square,\psi}(w,\t,\bar{\rho})$ is
$p$-torsion free, $R^{\square,\psi}(w,\t,\bar{\rho})[1/p]$ is reduced
and equidimensional. 

\item
for any finite extension $E'/E$, a map $x : R^{\square}(\bar{\rho}) \to
E'$ factors through $R^{\square,\psi}(w,\t,\bar{\rho})$ if and only if
the representation $\rho_x$ is of determinant $\psi\eps$,
potentially semi-stable of Hodge-Tate type
$w$, and of inertial type
$\t$.
\end{enumerate}
\end{theo}

\begin{rema}
\label{indepdet}
The ring  $R^{\square,\psi}(w,\t,\bar{\rho})$ can be non-zero only
if $w$, $\psi$ and $\t$ satisfy the equality:
$\psi|_{I_K} = (\det\t)\prod_{\tau}\eps_{\tau}^{n_{\tau}+2m_{\tau}}$
where $\eps_{\tau}$ is the Lubin-Tate character corresponding to $\tau$,
so that $\eps = \prod_{\tau}\eps_{\tau}$. In this case we say that $\psi$
is compatible with $\t$ and $w$. Note that by \cite[Lemma 4.3.1]{EG} 
the isomorphism class of $R^{\square,\psi}(w,\t,\bar{\rho})$ does not
depend on $\psi$ as long as it is compatible with $\t$ and $w$.
\end{rema}

\subsubsection{Irreducible components and extended types}

Suppose that $\t$ is a discrete series inertial type, and let $\t'$ be an
extended type such that $\t'|_{I_K}$ is isomorphic to $\t$. We define
a subset of the set of irreducible components of 
$\spec R^{\square,\psi}(w,\t,\bar{\rho})$ by saying that an irreducible
component is of type $\t'$ if:
\begin{enumerate}
\item
when $\t$ is of the form $\red$ or $\irr$, the irreducible component
has an $E'$-point $x$ with $\rho_x$ of extended type $\t'$
for some finite extension $E'/E$; 
\item
when $\t$ is of the form $\scal$, the irreducible component
has an $E'$-point $x$ with $\rho_x$ of
extended type $\t'$ that is not potentially crystalline.
\end{enumerate}

There can exist a component of type $\t'$ only if 
$\det \t' = \WD(\psi\eps)$, hence there are at most two
such extended types and then they are conjugate if $\t$ is of the form
$\scal$ or $\irr$, and at most one such extended type if $\t$ is of the
form $\red$. We say that $\t'$ is
compatible with $(\t,\psi)$ if $\t'|_{I_K}$ is isomorphic to $\t$ and $\det
\t' = \WD(\psi\eps)$. If $\t$ is of the form $\red$ and $\t'$ is
compatible to $(\t,\psi)$ then all irreducible components of $\spec
R^{\square,\psi}(w,\t,\bar{\rho})$ are of type $\t'$.

If a component is of type $\t'$, then for all closed points $x$ over a finite
extension $E'/E$, the representation $\rho_x$ is of type $\t'$. In
particular a component is of at most one extended type. This
follows from:

\begin{prop}
\label{family}
Let $\aa$ be an affinoid algebra which is a domain.
Let $\rho : G_K \to \GL_2(\aa)$ 
be a continuous $\aa$-linear representation
such that for any closed point $x \in \text{Max}(\aa)$, 
the representation $\rho_x$
is potentially semi-stable, with the same discrete series inertial type
$\t$, Hodge-Tate weights and determinant for all $x$.
If $\t$ is scalar and at least
one representation in the family is not potentially crystalline, or if
$\t$ is not scalar, then the extended type is constant in the family.
\end{prop}

\begin{proof}
Let $F$ be a finite extension of $K$ such that $\rho_x|_{G_F}$ is
semi-stable for all $x$. Such an $F$ exists and is determined by $\t$.
Using the results of \cite[Section 6.3]{BC} we can then construct a free 
$F_0\otimes\aa$-module $D_{st}(\rho)$ with a Frobenius $\phi$, a
monodromy operator $N$ and an action of $\Gal(F/K)$ that are $\aa$-linear
and such that for all
$x$, $\aa/\m_x\otimes_{\aa}D_{st}(\rho)$ is isomorphic to 
$D^F_{st}(\rho_x)$. We can apply the method of the construction of the
Weil-Deligne representation as given in \cite{CDT}, Appendix B.
to $D_{st}(\rho)$. This gives
a continuous representation $r : W_K \to \GL_2(\aa)$ and $N \in 
M_2(\aa)$ such that for all $x$, $(r_x,N_x)$ is the Weil-Deligne
representation attached to $\rho_x$. By assumption all representations $r_x$ have the
same restriction to inertia and the same determinant. If $\t$ is of the
form $\red$, this implies that the isomorphism class of $r_x$ is constant
and hence the extended type is constant in the family.

Suppose now that $\t$ is of the form $\irr$,
that is
$\t = (\ind_{I_L}^{I_K}\chi)|_{I_K}$ for some ramified quadratic
extension $L$ of $K$ and some character $\chi$ of $W_L$ that does not
extend to $W_K$. Then $\t|_{I_L} = \chi\oplus\chi'$ for characters
$\chi,\chi'$ that do not extend to $I_K$. Fix also an element $\alpha \in
I_K\setminus I_L$.
We can choose a basis $(e_1,e_2)$ of $\aa^2$
(after possibly replacing $\text{Max}(\aa)$ by some
admissible covering) such that for all $x$, we have $r_x(\beta)e_1 =
\chi(\beta)e_1$ and $r_x(\beta)e_2 = \chi'(\beta)e_2$ for all $\beta \in
I_L$, and $r_x(\alpha)e_1 = e_2$. 
Let $\Frob$ be any Frobenius element of $W_K$. Then the matrix in this
basis of $r_x(\Frob^2)$ and of $r_x(\beta)$ for any $\beta \in I_K$ is
constant, and the matrix of $r_x(\Frob)$ can take only two possible
values that determine the isomorphism class of $r_x$.
As the matrix varies continuously with $x$ and $\aa$ is
integral, it is constant, hence the isomorphism class of $r_x$ is
constant.

Suppose now that $\t$ is scalar. Let $\Frob$ be any Frobenius element of
$W_K$. Then the isomorphism class of $r_x$ is determined by the
characteristic polynomial of $r_x(\Frob)$.  Let $U$ be the Zariski open
subset of $\text{Max}(\aa)$ defined by the condition $N \neq 0$.  Then on
$U$ the eigenvalues of $r_x(\Frob)$ are of the form $\alpha_x$ and
$q\alpha_x$ for some $\alpha_x$.  As the determinant of $r_x$ is
constant, $\alpha_x^2$ is constant, hence the characteristic polynomial
of $r_x(\Frob)$ can take only two possible values on $U$. As $U$ is
Zariski dense in $\text{Max}(\aa)$ (because $\aa$ is a domain), it can
only take two possible values on $\text{Max}(\aa)$, and in fact only one
by continuity.  Hence the isomorphism class of $r_x$ is constant.
\end{proof}

\subsubsection{Deformation rings of fixed discrete series extended type}
\label{defdefrings}

We define a quotient $R^{\square,\psi}(w,\t',\bar{\rho})$ of
$R^{\square,\psi}(w,\t,\bar{\rho})$ by taking the maximal reduced quotient
supported on the set of irreducible components of
$\spec R^{\square,\psi}(w,\t,\bar{\rho})$ that are of type $\t'$.
We also define, following \cite[Section 5]{GG}, a ring
$R^{\square,\psi}(w,\t^{ds},\bar{\rho})$ corresponding to all the
irreducible components of \emph{some} extended type $\t'$. 

If $\t$ is of the form $\red$, there is exactly one extended type $\t'$
that is compatible with $(\t,\psi)$, and we have
$R^{\square,\psi}(w,\t',\bar{\rho}) = R^{\square,\psi}(w,\t,\bar{\rho}) =
R^{\square,\psi}(w,\t^{ds},\bar{\rho})$.

If $\t$ is of the form $\irr$, $R^{\square,\psi}(w,\t,\bar{\rho}) =
R^{\square,\psi}(w,\t^{ds},\bar{\rho})$, but there are two extended types
that are compatible with $(\t,\psi)$ so
$R^{\square,\psi}(w,\t',\bar{\rho})$ can be different from 
$R^{\square,\psi}(w,\t^{ds},\bar{\rho})$.

If $\t$ is scalar, $R^{\square,\psi}(w,\t^{ds},\bar{\rho})$ is a quotient of
$R^{\square,\psi}(w,\t,\bar{\rho})$ supported only on the components
containing points corresponding to representations that are not potentially crystalline
and is generally different from $R^{\square,\psi}(w,\t,\bar{\rho})$
(see also Lemma 5.5 of \cite{GG} and the remarks preceding it). If $\t'$
is an extended type compatible with $(\t,\psi)$, then
$R^{\square,\psi}(w,\t',\bar{\rho})$ is a quotient of
$R^{\square,\psi}(w,\t^{ds},\bar{\rho})$, but it can be different from it
as there are two possibilities for $\t'$.

Then we have the following properties:
\begin{prop}
\label{defextended}
\begin{enumerate}
\item
$R^{\square,\psi}(w,\t',\bar{\rho})$ is
$p$-torsion free, $R^{\square,\psi}(w,\t',\bar{\rho})[1/p]$ is reduced
and equidimensional. 

\item
for all finite extensions $E'/E$, if a map $x : R^{\square}(\bar{\rho}) \to
E'$ factors through $R^{\square,\psi}(w,\t',\bar{\rho})$ then
the representation $\rho_x$ is of determinant $\psi\eps$,
potentially semi-stable of Hodge-Tate type $w$
and of extended type $\t'$.

\item
for all finite extensions $E'/E$, a map $x : R^{\square}(\bar{\rho}) \to
E'$ such that the representation $\rho_x$ is of determinant
$\psi\eps$, potentially semi-stable of Hodge-Tate type $w$ and of
extended type $\t'$ factors through $R^{\square,\psi}(w,\t',\bar{\rho})$.
\end{enumerate}
\end{prop}

\begin{proof}
Properties $(1)$ and $(2)$
follow from the analogous properties for the inertial type $\t$,
and Proposition \ref{family}.

In the case not of the form $\scal$, property $(3)$ follows from the fact
that any $\rho_x$ of inertial type $\t$ is of some discrete series
extended type $\t'$. 

Suppose now that $\t$ is scalar, we can suppose that $\t$ is trivial. Let
$x$ be as in $(3)$: $x$ factors through
$R^{\square,\psi}(w,\t,\bar{\rho})$ by Theorem \ref{definertial} and
it defines a representation $\rho : G_K \to
\GL_2(\O_{E'})$ lifting $\bar{\rho}$ for some finite extension $E'$ of
$E$. If $\rho$ is not crystalline, then by definition $x$ factors through
$R^{\square,\psi}(w,\t',\bar{\rho})$. 
Suppose now that $\rho$ is
crystalline. Then $\WD(\rho)$ is of the form $(r,0)$ with $r$ isomorphic
to $\t'$. As $\t'$ is a discrete series extended type with trivial
restriction to inertia, this means that $r = \psi\otimes(1\oplus \|\cdot \|)$
for some unramified character $\psi$ of $W_K$.  In particular, we see
that there exists a nonzero map $\WD(\rho) \to
\WD(\rho)\otimes\|\cdot\|$. By Theorem D of \cite{All}, this means that
the point $x$ defining
the representation $\rho$ is a non-smooth point on 
$\spec R^{\square,\psi}(w,\t,\bar{\rho})[1/p]$. We know from the proof of
Lemma A.3 of \cite{Kis09a} that the union of the crystalline irreducible
components of $\spec R^{\square,\psi}(w,\t,\bar{\rho})[1/p]$ is smooth.
So this means that $x$ is a point of 
$\spec R^{\square,\psi}(w,\t,\bar{\rho})[1/p]$ which is also on an
irreducible component that contains non-crystalline points, and so is a
point on $\spec R^{\square,\psi}(w,\t',\bar{\rho})[1/p]$, and $x$ factors
through $R^{\square,\psi}(w,\t',\bar{\rho})$.
\end{proof}

\begin{rema}
\label{involution}
Suppose that $\bar{\rho}$ is irreducible. 
Then as $\bar{\rho}$ is isomorphic to
$\bar{\rho}\otimes\unr(-1)$, the map $\rho \mapsto
\rho\otimes\unr(-1)$ induces an involution of
$R^{\square,\psi}(w,\t^{ds},\bar{\rho})$ that exchanges 
$R^{\square,\psi}(w,\t',\bar{\rho})$ and
$R^{\square,\psi}(w,\t^{\prime\prime},\bar{\rho})$, where $\t'$ and
$\t^{\prime\prime}$ are the conjugate
extended types compatible with $(\t,\psi)$.
In particular 
$R^{\square,\psi}(w,\t',\bar{\rho})$ and
$R^{\square,\psi}(w,\t^{\prime\prime},\bar{\rho})$
are isomorphic.
\end{rema}

\begin{rema}
For an inertial type of the form $\car$, the extended type is not
constant on irreducible components.  In fact it follows from
\cite[Section 3.2]{GM}
that, in the case $K = \Q_p$, for an inertial type of this form there
exists only a finite number of isomorphism classes of potentially
semi-stable representations of given regular Hodge-Tate weights and
extended type.
\end{rema}

\section{Multiplicities}
\label{automorphic}

Let $D$ be the non-split quaternion algebra over $K$. In this section we
consider all smooth representations as having coefficients in
$\bar{\Q}_p$, unless otherwise specified.

\subsection{Local Langlands and Jacquet-Langlands}
\label{lljl}

We denote by $\JL$ the local Jacquet-Langlands correspondence, that
attaches to every irreducible smooth admissible representation of $D^{\times}$
a discrete series smooth representation of
$\GL_2(K)$ (that is, supercuspidal or a twist of the Steinberg representation).

We denote by $\rec$ the local Langlands correspondence that
attaches to each irreducible smooth admissible representation of $\GL_2(K)$ 
a Weil-Deligne representation of degree $2$ with the normalization
of \cite{HT01}, Introduction.

We set $\LL_D(\pi)  = (\rec \circ \JL(\pi))\otimes\|\cdot\|^{1/2}$, so that
the image of $\LL_D$ is exactly the discrete series Weil-Deligne
representations $(r,N)$ (see Paragraphs \ref{statements} and
\ref{galoismodform} for a justification of the normalization).  We give
some properties of $\LL_D$: let $\psi$ a character of $K^{\times}$, and
denote by $\Norm$ the reduced norm $D \to K$.  For $x \in
\bar{\Q}_p^{\times}$, we denote by $\unr_D(x)$ the character of
$D^{\times}$ given by $\unr(x) \circ \Norm$, and more generally for
$\psi$ a character of $K^{\times}$ we denote by $\psi_D$ the character
$\psi\circ \Norm$ of $D^{\times}$.  Then $\LL_D(\psi_D) =
(\psi\oplus\psi\|\cdot\|,N)$ with $N \neq 0$.

Let $\varpi_D$ be a uniformizer of $D$.
If $a \geq 1$, we set $U_D^a = 1+\varpi_D^a\O_D$. It is an open compact
subgroup of $D^{\times}$, which does not depend on the choice of
$\varpi_D$.
It follows from the
explicit description of smooth representations of $D^{\times}$ (as can be found
for example in Chapter 13 of \cite{BH06}) that any irreducible smooth
representation of $D^{\times}$ that is not a character has one of the
following forms:
\begin{enumerate}
\item
$\pi_D = \ind^{D^{\times}}_{L^{\times}U_D^a}\psi$ for some character $\psi$ and some ramified
quadratic extension $L$ of $K$.

\item
$\pi_D = \ind^{D^{\times}}_{L^{\times}U_D^a}\rho$ for some irreducible representation $\rho$
of $L^{\times}U_D^a$ of dimension $1$ or $q$ and $L$ the unramified quadratic extension of $K$.
\end{enumerate} 

\begin{prop}
\label{classredirr}
Let $r$ be a Weil-Deligne representation of dimension $2$ that is of
the form $\red$ or $\irr$ of Lemma \ref{classtypes}. Then the following
conditions are equivalent:
\begin{enumerate}
\item
the type of $r$ is of the form $\red$.
\item
$\LL_D^{-1}(r) \simeq \LL_D^{-1}(r)\otimes\unr_D(-1)$.
\item
$\LL_D^{-1}(r) = \ind^{D^{\times}}_{L^{\times}U_D^a}\rho$ for $L$ the unramified quadratic
extension of $K$, some $a$ and some representation $\rho$ of
$L^{\times}U_D^a$.
\item
the restriction of $\LL_D^{-1}(r)$ to $\O_D^{\times}$ is the sum of two irreducible
representations that differ by conjugation by $\varpi_D$.
\end{enumerate}
And the following conditions are equivalent:
\begin{enumerate}
\item
the type of $r$ of the form $\irr$.
\item
$\LL_D^{-1}(r) \not\simeq \LL_D^{-1}(r)\otimes\unr_D(-1)$.
\item
$\LL_D^{-1}(r) = \ind^{D^{\times}}_{L^{\times}U_D^a}\psi$ for some ramified quadratic
extension $L$ of $K$, some $a$ and some character $\psi$ of
$L^{\times}U_D^a$.
\item
the restriction of $\LL_D^{-1}(r)$ to $\O_D^{\times}$ is irreducible.
\end{enumerate}
\end{prop}

\begin{proof}
Note first that $\LL_D$ is compatible with twists by characters, as this
is the case for $\rec$ and $\JL$.

$(1) \Leftrightarrow (2)$ comes from Proposition \ref{two}.

$(1) \Leftrightarrow (3)$ comes from the explicit descriptions of the
local Langlands and Jacquet-Langlands correspondence 
(see \cite{BH06}).

$(3) \Leftrightarrow (4)$ is Proposition 3.8 of \cite{GG} (see also 
Sections $5$ and $6$ of \cite{Ger}).
\end{proof}

\subsection{Representations attached to a discrete series inertial type}
\label{reprtype}

\subsubsection{Representations of $D^{\times}$ and $\O_D^{\times}$}
\label{typeD}

Let $\t$ be some discrete series inertial type.  Let $(r,N)$ be some
discrete series Weil-Deligne representation with $\t = r|_{I_K}$.  Let
$\pi_{\t} = \LL_D^{-1}(r,N)$, which depends on the choice of $(r,N)$ only
up to unramified twist. If $(r,N)$ is of the form $\scal$ then $\pi_{\t}$ is
a character of $D^{\times}$, so the restriction of $\pi_{\t}$ to
$\O_D^{\times}$ is irreducible. As we have seen in Proposition
\ref{classredirr},
if $(r,N)$ is of the form $\irr$ then
$\pi_{\t}$ is still irreducible after restriction to $\O_D^{\times}$, and
if $(r,N)$ is of the form $\red$ then the restriction of $\pi_{\t}$ to
$\O_D^{\times}$ is the sum of two irreducible constituents that differ by
conjugation by $\varpi_D$.
Let $\sigma_D(\t)$ be one of the irreducible
constituents of the restriction of $\pi_{\t}$ to $\O_D^{\times}$; it
depends only on $\t$ and not on the choice of $(r,N)$ (this is the same
as the representation $\sigma_D(\t)$ of \cite{GG}, Section 3.1).  We can
recover $\pi_{\t}$ from $\sigma_D(\t)$ up to unramified twist. Hence we
have the following property:

\begin{prop}
\label{typesD}
Let $\pi_D$ be a smooth irreducible representation of $D^{\times}$.
Then $\Hom_{\O_D^{\times}}(\sigma_D(\t),\pi_D) \neq 0$ if and only if
$\LL_D(\pi_D)|_{I_K} \simeq \t$.
\end{prop}

\begin{rema}
As was already noted in \cite{GG},
contrary to the case of $\GL_2$, we see that the type for $D^{\times}$ is
not unique, at least for representations of the form $\red$. On the other
hand, as there are only one or two irreducible constituents for the
restriction to $\O_D^{\times}$ of a smooth irreducible representation of
$D^{\times}$, it is much easier to find a type.
\end{rema}

\subsubsection{The group $\G_{\varpi_K}$}
\label{GpiK}

Let $\varpi_K$ be a uniformizer of $K$ and $\varpi_D$ a uniformizer of
$D$ with $\varpi_D^2 = \varpi_K$.
Let $\G_{\varpi_K} = D^{\times}/\varpi_K^{\Z}$.
Then $\G_{\varpi_K}$ is isomorphic to the semi-direct product
$\O_D^{\times} \rtimes_{\varpi_D} \{1,\iota\}$, where the action of
$\iota$ on $\O_D^{\times}$ is by conjugation by $\varpi_D$. 
As a group, $\G_{\varpi_K}$ depends on $\varpi_K$, but 
not on the choice of $\varpi_D$ such that $\varpi_D^2 = \varpi_K$.  
Let $\xi=\unr_D(-1)$, that is the character of
$\G_{\varpi_K}$ that is trivial on $\O_D^{\times}$ and sends $\iota$ to
$-1$. 

Let $\t$ be a discrete series inertial type.
Using the representation $\pi_{\t}$ of Paragraph
\ref{typeD}, we can attach to $\t$ a smooth representation $\sigmaG(\t)$
of $\G_{\varpi_K}$, or equivalently a smooth representation of
$D^{\times}$ that is trivial on $\varpi_K$: there is an unramified twist
of $\pi_{\t}$ that is trivial on $\varpi_K$, as $\pi_{\t}(\varpi_K)$ is
scalar.

The relation between $\sigmaG(\t)$ and $\sigma_D(\t)$ is then given by
Proposition \ref{classredirr}.
If $\t$ is of the form $\scal$ or $\irr$ then the representation
$\sigmaG(\t)$ is defined only up to twist by $\xi$, 
and the restriction of $\sigmaG(\t)$
to $\O_D^{\times}$ is isomorphic to $\sigma_D(\t)$, which is irreducible.
If $\t$ is of the form
$\red$, then there is only one possibility for $\sigmaG(\t)$ and the
restriction of $\sigmaG(\t)$ to $\O_D^{\times}$ is isomorphic to the direct sum
of $\sigma_D(\t)$ and the representation $\sigma_D(\t)^{\varpi_D}$ 
obtained from $\sigma_D(\t)$ by conjugation by a uniformizer.

\begin{rema}
The representation $\sigmaG(\t)$ of $\G_{\varpi_K}$
is the analogue in our situation of the
representation $\sigma_{\tau'}$ of $\tilde{U}_0(\ell)$ in Section 1.2
of \cite{BCDT}.
\end{rema}

\subsubsection{Realizations of $\G_{\varpi_K}$}
\label{realization}

Let $K'$ be the unramified quadratic extension of $K$. By fixing an
embedding of $K'$ into $D$ and a basis of $D$ as a $K'$-vector space, we
can define an embedding $D^{\times} \to \GL_2(K')$, hence, after choosing
$K' \to \bar{\Q}_p$, an embedding $u: D^{\times} \to \GL_2(\bar{\Q}_p)$.
All such embeddings are conjugate 
in $\GL_2(\bar\Q_p)$ by the Skolem-Noether theorem.

Fix now $\varpi_K$ a uniformizer of $K$, $\varpi_D$ a square root of
$\varpi_K$ in $D$ and $\sqrt{\varpi_K}$ a square root of $\varpi_K$ in
$\bar{\Q}_p$. With these choices we can define an embedding $\tilde{u} :
\G_{\varpi_K} \to \GL_2(\bar{\Q}_p)$ by setting
$\tilde{u}|_{\O_D^{\times}} = u|_{\O_D^{\times}}$ and $\tilde{u}(\iota) =
\sqrt{\varpi_K}^{-1}u(\varpi_D)$.

Note that for each choice of $\varpi_K$ and $\sqrt{\varpi_K}$, all the
possible $\tilde{u}$ corresponding to the various choices of $u$ and
$\varpi_D$ are conjugate in $\GL_2(\bar{\Q}_p)$. Moreover for varying
choices of $\varpi_K$ all the $\tilde{u}|_{\O_D^{\times}}$ are conjugate. 

\subsection{Representations of $\Gamma_K$}

\subsubsection{The group $\Gamma_K$}
\label{GammaK}

Let $k$ be the residue field of $K$ and $\ell$ its quadratic extension,
so that $\O_D/\varpi_D \simeq \ell$. We define the group $\Gamma_K =
\ell^{\times} \rtimes \{1,\iota\}$ where $\iota$ acts on $\ell^{\times}$ by the non-trivial
$k$-automorphism of $\ell$. 

The quotient $\G_{\varpi_K}/(1+\varpi_D\O_D)$ is naturally isomorphic to
the group $\Gamma_K$, and the map $\G_{\varpi_K} \to \Gamma_K$ extends
the natural morphism $\O_D^{\times} \to \ell^{\times}$. As
$1+\varpi_D\O_D$ is a pro-$p$-group, any semi-simple representation of
$\G_{\varpi_K}$ in characteristic $p$ factors through $\Gamma_K$.

\subsubsection{Irreducible representations of $\Gamma_K$ in
characteristic $p$}

Let $\bar{\F}$ be an algebraic closure of $k$.
Fix $\ell \to \bar{\F}$, and let $q$ be the cardinality of $k$.

For an element $a$ in $\Z/(q^2-1)\Z$, we denote by $\chi_a$ the character
of $\ell^{\times}$ sending $x$ to $x^a$.

For an element $a$ in $\Z/(q-1)\Z$, we denote by $\delta_a$ the character of $\Gamma_K$ 
such that $\delta_a$ coincides with $\chi_{a(q+1)} = \chi_a|_{k^{\times}} \circ N_{\ell/k}$ 
on $\ell^{\times}$ and $\delta_a(\iota) = 1$.

Let $\xi$ be the character of $\Gamma_K$ that is trivial on $\ell^{\times}$ and
$\xi(\iota) = -1$.

Let $r_a = \ind_{\ell^{\times}}^{\Gamma_K}\chi_a$ for $a\in
\Z/(q^2-1)\Z$. Then $r_a$ is irreducible if and only if $a$ is not
divisible by $q+1$. If $a = (q+1)b$ then $r_a$ is isomorphic to 
$\delta_b\oplus\xi\delta_b$.

\begin{prop}
The irreducible representations of $\Gamma_K$ with coefficients in
$\bar{\F}$
are exactly the following:
\begin{itemize}
\item
The characters $\delta_a$ for $a \in \Z/(q-1)\Z$.

\item
The characters $\xi\delta_a$ for $a \in \Z/(q-1)\Z$.

\item
The representations $r_a$ for $a\in \Z/(q^2-1)\Z$ not divisible by $q+1$.
\end{itemize}

Moreover, these representations are all distinct, except for 
the relation $r_a = r_{qa}$. 
Finally, $\xi r_a$ is isomorphic to $r_a$, and $r_a|_{\ell^{\times}}
=\chi_a\oplus\chi_{qa}$.
\end{prop}

The irreducible representations of $\Gamma_K$ are the analogue in our
situation of the Serre weights.

\subsubsection{Reduction modulo $p$ of representations of $\G_{\varpi_K}$ attached to
discrete series types}
\label{reduction}

Let $\sigmaG(\t)$ be a representation
of $\G_{\varpi_K}$ attached to a discrete series inertial type $\t$ as in
Paragraph \ref{GpiK}. 
As $\G_{\varpi_K}$ is
compact, we can find an invariant lattice in $\sigmaG(\t)$, and consider the
semi-simplification of the reduction modulo $p$ of this representation. We denote by
$\bar{\sigmaG}(\t)$ the representation of $\Gamma_K$ that we obtain (it is
semi-simple, independent of the choice of the invariant lattice and its
restriction to $\ell^{\times}$ is independent of any choice). 

\begin{prop}
\label{liftistype}
Each irreducible representation of $\Gamma_K$ over $\bar{\F}$ has a lift
in characteristic $0$ that is of the form $\sigmaG(\t)$ for some discrete
series inertial type $\t$.
\end{prop}

\begin{proof}
Let $\delta$ be an irreducible representation of $\Gamma_K$ of dimension $1$.
It is of the form $\chi  \circ N_{\ell/k}$ or $\xi\chi  \circ N_{\ell/k}$
for some character $\chi$ of $k^{\times}$.
We define a scalar inertial type $\t_{\delta}$ 
by $\t_{\delta} = (\tilde{\chi}\oplus\tilde{\chi})|_{I_K}$, where
$\tilde{\chi}$ denotes the image by local class field theory of the
Teichmüller lift of the character $\chi \circ ({K}^{\times} \to
k^{\times})$.  Then we can choose $\sigmaG(\t_{\delta})$ so that
$\bar{\sigmaG}(\t_{\delta})$ is isomorphic to $\delta$ (note that
$\sigmaG(\t_{\delta})$ depends on $\delta$ and not only on
$\t_{\delta}$).

Let $r$ be an irreducible representation of $\Gamma_K$ of dimension $2$.
There exists $a \in \Z/(q^2-1)\Z$ not divisible by $q+1$
such that $r|_{\ell^{\times}} = \chi_a\oplus\chi_{qa}$. Let $K'$ be the
unramified quadratic extension of $K$, and $\tilde{\chi}_a : W_{K'} 
\to \bar{\Q}_p^{\times}$ the tame character given by the Teichmüller lift
of $\chi_a \circ ({K'}^{\times} \to \ell^{\times})$. 
We define an inertial type $\t_r$ of the form $\red$ by
$\t_r = (\ind_{W_{K'}}^{W_K}\tilde{\chi}_a)|_{I_K}$.
Then $\bar{\sigmaG}(\t_r)$ is isomorphic to $r$, as follows from the
explicit constructions in Chapter 13 of \cite{BH06}.
\end{proof}

We denote by $\R(\Gamma_K)$ the Grothendieck ring of
representations of $\Gamma_K$ with coefficients in $\bar{\F}$.
We denote by $[\sigma]$ the image in $\R(\Gamma_K)$ of a representation $\sigma$
of $\Gamma_K$.

We now compute the reduction of some representations of $\G_{\varpi_K}$
attached to discrete series types.
Let $L$ be the ramified quadratic extension of $K$ generated by the
square root of $\varpi_K$, and fix an embedding of $L$ into $D$. 
Any semi-simple representation of $L^{\times}U^a_D$ in characteristic
$p$ is trivial on $U_D^1\cap L^{\times}U^a_D$ as this is a pro-$p$-group
(recall that $U_D^a$ was defined
in Paragraph \ref{lljl}).
Any representation of $L^{\times}U^a_D$
that is trivial on the subgroup generated by $\varpi_K$
factors through the image of $L^{\times}U^a_D$ inside $\G_{\varpi_K}$.
Hence any semi-simple representation of $L^{\times}U^a_D$ in
characteristic $p$ that is trivial on $\varpi_K$ 
factors through 
the subgroup $L^{\times}U_D^a/\langle U_D^1,\varpi_K^{\Z}\rangle$ of
$\Gamma_K$. Let us call this subgroup $\Delta$, it is equal to
$k^{\times}\times\{1,\iota\} \subset \ell^{\times}\rtimes \{1,\iota\}$.

\begin{prop}
\label{reductypeirr}
Let $L$ be as above and let $\theta =
\ind^{D^{\times}}_{L^{\times}U^a_D}\psi$ for some smooth character $\psi$
of $L^{\times}U^a_D$ with $\psi(\varpi_K)=1$. Then $\theta$ factors through
$\G_{\varpi_K}$. Denote by $\bar{\theta}$ the semi-simple representation of
$\Gamma_K$ which is the reduction modulo $p$ of $\theta$. 
Let $n \in \Z/(q-1)\Z$ and $\alpha \in \{0,1\}$ be such that $\bar{\psi} =
\xi^{\alpha}\chi_n|{\Delta}$ as a representation of $\Delta$. We denote
by $I(\bar{\psi})$ the set of irreducible representations of $\Gamma_K$
with central character equal to $\bar{\psi}|_{k^{\times}}$.
Then we have in $\R(\Gamma_K)$:
\begin{enumerate}
\item 
if $\psi(-1)=-1$, that is, $n$ is odd, then $I(\bar{\psi})$ consists of
$(q+1)/2$ representations of dimension $2$, and:
\[
[\bar{\theta}] = q^{a-1}\left(\sum_{r\in I(\bar{\psi})}[r]\right)
\]
\item 
if $\psi(1)=1$, that is, $n$ is even, then $I(\bar{\psi})$ consists
of $(q-1)/2$ representations of dimension $2$ and $4$ representations of
dimension $1$, and:
\[
[\bar{\theta}] = q^{a-1}\left(\sum_{\substack{r\in I(\bar{\psi})\\
\dim(r)=2}}
[r]\right)
+ \frac{q^{a-1}+1}{2} [\xi^{\alpha}]([\delta_{n/2}] +
[\delta_{(n+q-1)/2}])
+ \frac{q^{a-1}-1}{2} [\xi^{\alpha+1}]([\delta_{n/2}] +
[\delta_{(n+q-1)/2}])
\]
\end{enumerate}
\end{prop}

\begin{proof}
We proceed as in \cite{BD}, Section $4$. 
We have that
$[L^{\times}U^1_D:L^{\times}U^a_D]= q^{a-1}$ (note that the essential
conductor of $\theta$ is $2a+1$).
The reduction modulo $p$ of 
$\ind_{L^{\times}U_D^a}^{L^{\times}U_D^1}\psi$ is 
the sum of $\frac{q^{a-1}+1}{2}$ copies of $\bar{\psi}$
and of $\frac{q^{a-1}-1}{2}$ copies of
$\xi\bar{\psi}$.
Let $\mu$ be a smooth character of $L^{\times}U_D^1$ in characteristic $p$ with
$\mu(\varpi_K)=1$, then $\ind_{L^{\times}U_D^1}^{\G_{\varpi_K}}\mu$
factors through $\Gamma_K$, and the representation of $\Gamma_K$ that we
obtain is $\ind_{\Delta}^{\Gamma_K}\mu$,
which can be computed via Brauer characters.
\end{proof}

It follows from Proposition \ref{classredirr} that
Proposition \ref{reductypeirr} gives $\bar\sigmaG(\t)$
when $\t$ is of type $\irr$ under some compatibility condition between
$\t$ and $\varpi_K$. As we will see in Paragraph \ref{statements} this
compatibility condition is harmless. When $\t$ is scalar,
$\bar\sigmaG(\t)$ is easy to compute as $\sigmaG(\t)$ is of dimension $1$. 
The value of $\bar\sigmaG(\t)$ when $\t$ is of the form $\red$
could be immediately obtained from Proposition 4.6 of \cite{BD}: as
$\bar{\sigmaG}(\t) = \xi\bar{\sigmaG}(\t)$, it is entirely determined by
its restriction to $\ell^{\times}$.
We do not give details as they are not really needed. Indeed, our goal in
computing $\bar\sigmaG(\t)$ is to allow us to compute the multiplicity
of some deformation rings as we shall see in Theorems \ref{mainQp} and
\ref{mainK}, but for $\t$ of the form $\red$ this multiplicity can be
computed by the formula coming from the Breuil-Mézard conjecture for
$\GL_2$.
Note that complete results and computations can be found in
\cite{Tok}.

\subsection{A reformulation of a result of Gee and Geraghty}
\label{reformulation}

Let $\Gamma = \Gamma_K$. We fix $\varpi_K$ a uniformizer of $K$.
We denote by $w_0$ the Hodge-Tate type $(0,0)_{\tau\in
\Hom(K,\bar{\Q}_p)}$.

In the case $K = \Q_p$, let $w = (n,m)$ be a Hodge-Tate type. 
We set $|w| = n+2m$.
We define a representation
$\sigma_{w} = \Sym^{n}\otimes\det^{m}$ of $\GL_2(\bar{\Q}_p)$, hence of
$\G = \G_{\varpi_{\Q_p}}$ via a realization of $\G$ as in Paragraph
\ref{realization}. 
In particular $\sigma_{w_0}$ is the trivial representation of $\G$.
The isomorphism class of the restriction of $\sigma_w$ to $\O_D^{\times}$
does not depend on the particular choice of a realization, as they are
all conjugate in restriction to $\O_D^{\times}$.  
We can see the reduction modulo $p$ of $\sigma_w$ as a representation of
$\Gamma$ by restriction. We denote its image in $\R(\Gamma)$ by
$\bar{\sigma_w}$, it does not depend on any choices made (including
$\varpi_K$): indeed it is the restriction to $\Gamma$ of the
representation $\Sym^{n}\otimes\det^{m}$ of $\GL_2(\F)$ via any embedding
of $\Gamma$ into $\GL_2(\F)$, and all such embeddings are conjugate.

We denote by $\pi$ a uniformizer of the field $E$ of Paragraph
\ref{defring}. For any noetherian local
ring $A$, we denote by $e(A)$ the Hilbert-Samuel multiplicity $e(A,A)$
(see \cite{Mat} for the definition of the Hilbert-Samuel multiplicity,
and also \cite{Kis09a}, Section 1.3 for properties relevant to our
situation).

Let $\ell = \F_{q^2}$, and let $\R(\ell^\times)$ be the Grothendieck ring of
representations of $\ell^\times$ with coefficients in $\bar{\F}$.

For $K = \Q_p$, we recall the following result
(Corollary 5.7 of \cite{GG}), 
which is the consequence of the main result of
\cite{GG} and the usual formulation of the Breuil-Mézard conjecture
proved in \cite{Kis09a}, \cite{Pas} and \cite{HT15}.
Here $\sigma_D(\t)$ is, as in
Paragraph \ref{typeD}, a choice of irreducible sub-representation of the
restriction of $\pi_{\t}$ to $\O_D^{\times}$:

\begin{theo}
\label{BMquatQp}
Let $\bar{\rho} : G_{\Q_p} \to \GL_2(\F)$, and suppose that $p \geq 5$ if 
$\bar\rho$ is a twist of an extension of the trivial
representation by the cyclotomic character.
There exists a positive linear functional $i_{D,\bar{\rho}} : \R(\ell^{\times})
\to \Z$ such that for each discrete series inertial type $\t$, and each
choice of $\sigma_D(\t)$, we have
$e(R^{\square,\psi}(w,\t^{ds},\bar{\rho})/\pi) =
i_{D,\bar{\rho}}([\bar{\sigma_D}(\t)\otimes\bar{\sigma_w}|_{\ell^{\times}}])$.
\end{theo}

We return to the case of a general $K$.
We have the following well-known result (see for example \cite{GS}, Lemma
3.5):

\begin{prop}
\label{ordinary}
Let $\bar{\rho}$ be a continuous representation of $G_{K}$ of dimension $2$
with coefficients in $\F$.
Suppose that $\bar{\rho}$ has a potentially semi-stable lift with scalar
type $\t = \psi\oplus\psi$ and Hodge-Tate weights
$(0,1)_{\tau\in\Hom(K,\bar{\Q}_p)}$ which is not
potentially crystalline. Then $\bar{\rho}$ is an unramified twist of
$\smat {\omega}*01 \otimes \bar{\psi}$.  
\end{prop}

We deduce from this that when $\bar{\rho}$ is not a twist of
an extension of the trivial character by the cyclotomic character,
$R^{\square,\psi}(w_0,\t^{ds},\bar{\rho})
= R^{\square,\psi}_{\text{cr}}(w_0,\t^{ds},\bar{\rho})$
for any discrete series type $\t$, where the second ring
parametrizes only representations that are potentially crystalline.
Hence we can deduce from the main result of \cite{GG} and  \cite[Theorem
A]{GK} the
following:

\begin{theo}
\label{BMquatK}
Let $\bar{\rho} : G_{K} \to \GL_2(\F)$ a continuous representation that
is not a twist of an extension of the trivial
representation by the cyclotomic character.
There exists a positive linear functional $i_{D,\bar{\rho}} : \R(\ell^{\times})
\to \Z$ such that for each discrete series inertial type $\t$, and each
choice of $\sigma_D(\t)$, we have
$e(R^{\square,\psi}(w_0,\t^{ds},\bar{\rho})/\pi) =
i_{D,\bar{\rho}}([\bar{\sigma_D}(\t)])$.
\end{theo}

Let $d_{\t} = 1$ if $\t$ has the form $\scal$ or $\irr$, and $d_{\t}=2$ if 
$\t$ has the form $\red$, so that $d_{\t}$ is the number of irreducible
components of $\sigmaG(\t)|_{\O_D^{\times}}$. Then we can give a reformulation
of Theorems \ref{BMquatQp} and \ref{BMquatK} in terms of representations of $\Gamma$:

\begin{theo}
\label{BMquatbisQp}
Let $\bar{\rho} : G_{\Q_p} \to \GL_2(\F)$, and suppose that $p \geq 5$ if 
$\bar\rho$ is a twist of an extension of the trivial
representation by the cyclotomic character.
There exists a positive linear functional $i_{\bar{\rho}} : \R(\Gamma) \to
\Z$ such that for each discrete series inertial type $\t$ and each choice of
$\sigmaG(\t)$ we have
$d_{\t}e(R^{\square,\psi}(w,\t^{ds},\bar{\rho})/\pi) =
i_{\bar{\rho}}([\bar{\sigmaG}(\t)\otimes\bar{\sigma_w}])$.
\end{theo}

\begin{theo}
\label{BMquatbisK}
Let $\bar{\rho} : G_{K} \to \GL_2(\F)$ a continuous representation that
is not a twist of an extension of the trivial
representation by the cyclotomic character.
There exists a positive linear functional $i_{\bar{\rho}} : \R(\Gamma) \to
\Z$ such that for each discrete series inertial type $\t$ and each choice of
$\sigmaG(\t)$ we have
$d_{\t}e(R^{\square,\psi}(w_0,\t^{ds},\bar{\rho})/\pi) =
i_{\bar{\rho}}([\bar{\sigmaG}(\t)])$.
\end{theo}

\begin{proof}[Proof of Theorems \ref{BMquatbisQp} and \ref{BMquatbisK}]
It follows from Theorems \ref{BMquatQp} (resp. \ref{BMquatK}) 
and the definition of $\sigmaG(\t)$
that we have the equality
$d_{\t}e(R^{\square,\psi}(w,\t^{ds},\bar{\rho})/\pi) =
i_{D,\bar{\rho}}([\bar{\sigmaG}(\t)\otimes\bar{\sigma_w}|_{\ell^{\times}}])$.
Set $i_{\bar{\rho}}([\gamma]) = i_{D,\bar{\rho}}([\gamma|_{\ell^{\times}}])$ for all
irreducible representations $\gamma$ of $\Gamma$ to get the result.
\end{proof}

In particular, we observe that $i_{\bar{\rho}}([\gamma]) =
i_{\bar{\rho}}([\xi\gamma]) = 
(\dim \gamma)e(R^{\square,\psi}(w_0,\t_{\gamma}^{ds},\bar{\rho})/\pi)$ for any
irreducible representation $\gamma$ of $\Gamma$, where $\t_{\gamma}$ is
the inertial type defined in the proof of Proposition \ref{liftistype}.

We denote by $\WG(\bar{\rho})$ the set of $\gamma$ such that
$i_{\bar{\rho}}([\gamma]) \neq 0$.  This is the translation in the setting
of representations of $\Gamma$ of the (predicted) quaternionic Serre
weights of \cite{GS}. Note in particular that, as in \cite{GS}, the set
$\WG(\bar{\rho})$ is determined by the existence of certain lifts of
$\bar{\rho}$ that have all their Hodge-Tate weights equal to $(0,1)$, 
which makes the situation with quaternion algebras simpler than the
situation of Serre weights for $\GL_2$, since, for $\GL_2$, one cannot in
general lift a Serre weight as a type in characteristic $0$.

\subsection{Multiplicity formulas}
\label{statements}

We now state our main theorems for the multiplicity of the special fiber
of the discrete series extended type deformation rings, which we prove in
Section \ref{proof}. 

\begin{theo}
\label{mainQp}
Let $\bar{\rho}$ be a continuous representation of $G_{\Q_p}$ of dimension $2$
with coefficients in $\F$.
Suppose that $p \geq 5$ if $\bar{\rho}$ is a twist of an extension of
the trivial character by the cyclotomic character.
Let $\bar{\psi} = \omega^{-1}\det\bar{\rho}$, $\varpi_{\Q_p}$ a
uniformizer of $\Q_p$ and $\alpha$ a square root of
$\bar{\psi}(\varpi_{\Q_p})^{-1}$.

There exists a positive linear form $\mu_{\bar{\rho}}$ on $\R(\Gamma)$
with values in $\Z$ satisfying the following property:
for any discrete series inertial type $\t$, Hodge-Tate type $w$,
character $\psi$ lifting $\bar{\psi}$ compatible with $\t$ and $w$, 
and $\t^+$ a discrete series
extended type compatible with $(\t,\psi)$ we have:
\[
e(R^{\square,\psi}(w,\t^+,\bar{\rho})/\pi) =
\mu_{\bar{\rho}}([\bar{\sigmaG}(\t)\otimes\bar{\sigma_{w}}])
\]
for the choice of representation $\sigmaG(\t)$ of $\G_{\varpi_{\Q_p}}$ such that
$\t^+ = \LL_D(\sigmaG(\t))\otimes\unr(a\varpi_{\Q_p}^{|w|})^{-1}$ for
some $a \in \z^{\times}$ lifting $\alpha$.
\end{theo}

\begin{theo}
\label{mainK}
Let $\bar{\rho}$ be a continuous representation of $G_{K}$ of dimension $2$
with coefficients in $\F$ that is not a twist of an extension of the
trivial character by the cyclotomic character.
Let $\bar{\psi} = \omega^{-1}\det\bar{\rho}$, $\varpi_{K}$ a
uniformizer of $K$ and $\alpha$ a square root of
$\bar{\psi}(\varpi_{K})^{-1}$.

There exists a positive linear form $\mu_{\bar{\rho}}$ on $\R(\Gamma)$
with values in $\Z$ satisfying the following property:
for any discrete series inertial type $\t$,
character $\psi$ lifting $\bar{\psi}$ compatible with $\t$ and $w_0$, 
and $\t^+$ a discrete series
extended type compatible with $(\t,\psi)$ we have:
\[
e(R^{\square,\psi}(w_0,\t^+,\bar{\rho})/\pi) =
\mu_{\bar{\rho}}([\bar{\sigmaG}(\t)])
\]
for the choice of representation $\sigmaG(\t)$ of $\G_{\varpi_K}$ such that
$\t^+ = \LL_D(\sigmaG(\t))\otimes\unr(a)^{-1}$ for
some $a \in \z^{\times}$ lifting $\alpha$.
\end{theo}

\begin{rema}
\label{multiconjugate}
It follows from the definition of the compatibility of $\t^+$ with
$(\t,\psi,w)$ that there exists indeed a choice of $\sigmaG(\t)$
satisfying the condition. If $\t^-$ is the extended type conjugate to
$\t^+$, then the choices of $\sigmaG(\t)$ for $\t^+$ and $\t^-$ differ by
multiplication by $\xi$.

In the case when $\t$ is of the form $\red$,  
recall that there is only one extended type $\t^+$ compatible
with $(\t,\psi)$, and $R^{\square,\psi}(w,\t^{ds},\bar{\rho}) = 
R^{\square,\psi}(w,\t^{+},\bar{\rho})$.
There is no choice to be made for ${\sigmaG}(\t)$ as it is isomorphic
to $\xi{\sigmaG}(\t)$.
\end{rema}

We have the following proposition, which is a consequence of Proposition 1.3.9 of
\cite{Kis09a}:

\begin{prop}
\label{summult}
Let $\t^+$, $\t^-$ be the two distinct conjugate extended types compatible 
with $(\t,\psi)$ with $\t$ of the form $\scal$ or $\irr$.  Then 
$e(R^{\square,\psi}(w,\t^+,\bar{\rho})/\pi) + 
e(R^{\square,\psi}(w,\t^-,\bar{\rho})/\pi) = 
e(R^{\square,\psi}(w,\t^{ds},\bar{\rho})/\pi)$.
\end{prop}

We have the following corollary ($\t_r$ and $\t_{\delta}$ are the
inertial types defined in the proof of Proposition \ref{liftistype}):
\begin{coro}
$\mu_{\bar{\rho}}([r]) =
e(R^{\square,\psi}(w_0,\t_r^{ds},\bar{\rho})/\pi)$
for any irreducible representation $r$ of $\Gamma$ of dimension $2$, and
$\mu_{\bar{\rho}}([\delta]+[\xi\delta]) =
e(R^{\square,\psi}(w_0,\t_{\delta}^{ds},\bar{\rho})/\pi)$
for any irreducible representation $\delta$ of $\Gamma$ of dimension $1$.
In particular, 
for any irreducible representation $\gamma$ of $\Gamma$,
we have $\mu_{\bar{\rho}}([\gamma]) + \mu_{\bar{\rho}}([\xi\gamma]) = 
i_{\bar{\rho}}([\gamma])$.
\end{coro}

\begin{proof}
Let $r$ be an irreducible representation of $\Gamma$ of dimension $2$.
Then $r = \bar\sigmaG(\t_r)$ for some inertial type $\t_r$ of the form
$\red$ by Proposition \ref{liftistype}.
Then as remarked in Paragraph \ref{defdefrings}, if $\t_r^+$ is the
extended type compatible with $(\t_r,\psi)$, then 
$R^{\square,\psi}(w_0,\t_r^+,\bar{\rho}) =
R^{\square,\psi}(w_0,\t_r^{ds},\bar{\rho})$, hence the formula in this
case.
Let $\delta$ be an irreducible representation of $\Gamma$ of dimension
$1$.
Then 
$e(R^{\square,\psi}(w_0,\t_{\delta}^{ds},\bar{\rho})/\pi) =
e(R^{\square,\psi}(w_0,\t_\delta^{+},\bar{\rho})/\pi) +
e(R^{\square,\psi}(w_0,\t_\delta^{-},\bar{\rho})/\pi)$ where $\t_\delta^{+}$ and
$\t_\delta^{-}$ are the two conjugate extended types compatible with
$(\t_\delta,\psi)$. So we deduce the formula from Remark
\ref{multiconjugate}.
The formula with $i_{\bar\rho}$ then follows from Theorems
\ref{BMquatbisQp} and \ref{BMquatbisK}.
\end{proof}

It follows from this corollary 
that $\mu_{\bar{\rho}}([\gamma]) = 0$ if $\gamma \not \in \WG(\bar{\rho})$.
We begin the definition of $\mu_{\bar{\rho}}$
by setting $\mu_{\bar{\rho}}([\gamma]) = 0$ for any irreducible $\gamma$ 
not in $\WG(\bar{\rho})$. With this definition, the equalities of Theorem
\ref{mainQp} and \ref{mainK} hold
for all $\t,\psi,w$ (with $w=w_0$ if $K \neq \Q_p$)
such that $R^{\square,\psi}(w,\t^{ds},\bar{\rho}) = 0$.

From Proposition \ref{ordinary} we deduce:
\begin{prop}
\label{dim1}
If $\bar{\rho}$ is a representation such that $\WG(\bar{\rho})$
contains a representation $\delta$ of dimension $1$ then $\bar{\rho}$ is a twist
of an extension of the trivial character by the cyclotomic character and
there is at most one possible value for $\delta$ for which
$\mu_{\bar{\rho}}(\delta) \neq 0$.
\end{prop}

\begin{rema}
When $K=\Q_p$, $\bar{\rho}$ is a twist
of an extension of the trivial character by the cyclotomic character if
and only if $\WG(\bar{\rho})$
contains a representation $\delta$ of dimension $1$, and then
$i_{\bar{\rho}}(\delta) =1$. This follows from the explicit computations
of deformation rings that can be found in \cite[Section 5.2]{BM}.
\end{rema}

\begin{prop}
\label{samevaluemu}
If $\bar{\rho}$ is
not a twist of an extension of the trivial character by the cyclotomic
character then for any representation $\gamma$ of $\Gamma$ we have
$\mu_{\bar\rho}([\gamma]) = \mu_{\bar\rho}([\xi\gamma])$.
\end{prop}

\begin{proof}
It suffices to prove this for representations $\gamma$ that are
irreducible. If $\dim\gamma = 2$ then $\xi\gamma = \gamma$ so the
statement holds. If $\dim\gamma = 1$ then by Proposition \ref{dim1} both
sides of the equality are zero.
\end{proof}

\begin{coro}
\label{samemult}
Let $\bar{\rho}$ be a continuous representation of $G_{K}$ of dimension $2$
with coefficients in $\F$ which is
not a twist of an extension of the trivial character by the cyclotomic
character.
Then for any discrete series inertial type $\t$,
and any Hodge-Tate type $w$ if $K=\Q_p$, or for $w=w_0$ if $K \neq \Q_p$, 
we have
$e(R^{\square,\psi}(w,\t^+,\bar{\rho})/\pi) =
e(R^{\square,\psi}(w,\t^-,\bar{\rho})/\pi) =
\frac{d_{\t}}{2}e(R^{\square,\psi}(w,\t^{ds},\bar{\rho})/\pi)$.
\end{coro}

\begin{proof}
The first equality comes from Proposition \ref{samevaluemu} and Remark
\ref{multiconjugate}.
The last equality follows from Proposition \ref{summult}.
\end{proof}

\begin{rema}
If $\bar{\rho}$ is irreducible Corollary \ref{samemult} 
holds even without Theorems \ref{mainQp} and \ref{mainK}, 
because of Remark \ref{involution}.
\end{rema}

\begin{coro}
Let $K = \Q_p$.
Suppose that there exists a representation $\delta$ of dimension $1$ of $\Gamma$ with
$i_{\bar{\rho}}(\delta)\neq 0$ (that is, $\bar{\rho}$ is a twist of an extension 
of the trivial character by the cyclotomic character, and then
$i_{\bar{\rho}}(\delta) = 1$).
Then for any discrete series inertial type $\t$, Hodge-Tate type $w$,
character $\psi$ and pair of conjugate extended types $(\t^+,\t^-)$
compatible with $(t,\psi)$, we have either 
$e(R^{\square,\psi}(w,\t^+,\bar{\rho})/\pi) =
e(R^{\square,\psi}(w,\t^-,\bar{\rho})/\pi)$, or 
$|e(R^{\square,\psi}(w,\t^+,\bar{\rho})/\pi) -
e(R^{\square,\psi}(w,\t^-,\bar{\rho})/\pi)| = 1$.
The former takes place in particular when $R^{\square,\psi}(w,\t^{ds},\bar{\rho})=0$,
or $\t$ is of the form $\red$, or $\t$ is of the form $\irr$ with
$\pi_{\t}$ of the form $\ind^{D^{\times}}_{L^{\times}U_D^a}\psi$ for some ramified quadratic
extension $L$ of $K$, some $a$ and some character $\psi$ of
$L^{\times}U_D^a$ with $\psi(-1)=-1$ (see Proposition \ref{classredirr}
for the notations).
\end{coro}

We see examples where we have
$|e(R^{\square,\psi}(w,\t^+,\bar{\rho})/\pi) -
e(R^{\square,\psi}(w,\t^-,\bar{\rho})/\pi)| = 1$
in Section \ref{appli}.
 
\begin{proof}
Note that we can choose $\varpi_{\Q_p}$ as we wish to compute the
multiplicities.
Let $\sigmaG(\t)$ be a choice of representation attached to $\t$ as in
Paragraph \ref{GpiK}. We need to compute
$[\bar{\sigmaG(\t)}\otimes\bar{\sigma_w}:\delta] -
[\xi\bar{\sigmaG(\t)}\otimes\bar{\sigma_w}:\delta]$. 
We do this using the results of Proposition \ref{reductypeirr} and the remarks
that follow for $\bar{\sigmaG(\t)}$, and the Lemma below for
$\bar{\sigma_w}$.
\end{proof}

\begin{lemm}
\label{decsym}
In $\R(\Gamma_{\Q_p})$ we have that
$[\det^m] = [\xi^m\delta_m]$ for all $m$ 
and 
$[\Sym^{2n}\F^2] = 
[\delta_n] + \sum_{i=1}^{n}[r_{n(p+1)+i(p-1)}]$
and 
$[\Sym^{2n+1}\F^2] = 
\sum_{i=0}^{n}[r_{n(p+1)+p+i(p-1)}]$ for all $n \geq 0$.

Moreover $r_{n(p+1)+i(p-1)}$ is irreducible for 
$0 < i < (p+1)/2$
and
$(p+1)/2 < i < p+1$ 
and 
$r_{n(p+1)+(p-1)(p+1)/2} = 
\delta_{n+(p-1)/2}\oplus\xi\delta_{n+(p-1)/2}$ 
and 
$r_{n(p+1)} = \delta_n\oplus \xi\delta_n$.
\end{lemm}

\bigskip

Our proof of Theorems \ref{mainQp} and \ref{mainK} is by deducing them
from the usual version of the Breuil-Mézard conjecture, in the cases
where it is already known. We can hope that this method generalizes to
the cases of the Breuil-Mézard conjecture that are not yet known, which
leads us to the following:

\begin{conj}
\label{conj}
Let $\bar{\rho}$ be a continuous representation of $G_{K}$ of dimension $2$
with coefficients in $\F$.
There exists a positive linear form $\mu_{\bar{\rho}}$ on $\R(\Gamma)$
with values in $\Z$ satisfying the following property:
for any discrete series inertial type $\t$,
Hodge-Tate type $w$, character $\psi$ lifting
$\omega^{-1}\det\bar{\rho}$ compatible with $\t$ and $w$, 
and extended type
$\t^+$ compatible with $(\t,\psi)$,
there exists a choice of representation $\sigmaG(\t)$ of $\G$
such that we have:
\[
e(R^{\square,\psi}(w,\t^+,\bar{\rho})/\pi) =
\mu_{\bar{\rho}}([\bar{\sigmaG}(\t)\otimes\bar{\sigma_w}])
\]
When $\bar{\rho}$ is not a twist of an extension of the trivial character
by the cyclotomic character, we have 
$e(R^{\square,\psi}(w,\t^+,\bar{\rho})/\pi) =
e(R^{\square,\psi}(w,\t^-,\bar{\rho})/\pi)$ where $\t^-$ is the extended
type that is conjugate to $\t^+$.
\end{conj}

\section{Quaternionic modular forms}
\label{quatmodforms}

\subsection{Global setting}
\label{globalsetting}

Let $F$ be a totally real number field such that for all places $v
\mid p$, $F_v$ is isomorphic to $K$. We denote by $\Sigma_p$ the set of
places above $p$, and we assume that the number of infinite places of $F$ 
has the same parity as the
cardinality of $\Sigma_p$.
Let $B$ be the quaternion algebra with center $F$ that is
ramified exactly at the infinite places of $F$ and at $\Sigma_p$, which
exists thanks to the parity condition.

For all $v \in \Sigma_p$, we fix an isomorphism between $B_v$ and the
quaternion algebra $D$ of Section \ref{automorphic}. 
For any finite place $v$
of $F$ that is not in $\Sigma_p$, fix an isomorphism between $B_v$ and
$M_2(F_v)$ so that $\O_{B_v}^{\times}$ corresponds to $\GL_2(\O_{F_v})$.
We fix $v_0 \in \Sigma_p$ and denote $\Sigma_p\setminus\{v_0\}$ by
$\Sigma_p'$.

Let $\varpi_K$ be a uniformizer of $K$.
We denote by $\G$ the group $\G_{\varpi_K}$ of Paragraph \ref{GpiK}.
We fix a uniformizer $\varpi_D$ of $D$ with $\varpi_D^2 = \varpi_K$.

\subsection{Modular forms}
\label{modular}
We recall the theory of quaternionic modular forms (see for example
\cite[Section 1]{Tayb}, and also \cite[Section 4.1]{Kha} 
and \cite[Section 2]{GS} for
the situation with a quaternion algebra ramified at $p$).

Denote by $\A_F^f \subset \A_F$ the ring of finite adeles of $F$.
Let $U = \prod_vU_v$ 
be a compact open subgroup of $(B \otimes_F \A_F^f)^{\times}$ such that for
all finite places $v$, $U_v \subset \O_{B_v}^{\times}$, and for all $v \in
\Sigma_p$, $U_v = \O_{B_v}^{\times}$.

Let $A$ be a topological $\Z_p$-algebra. For all $v \mid p$,
let $(\sigma_v,V_v)$ be a representation of $U_v$ on a finite free
$A$-module.
We define a representation $\sigma$ of $U$ on $V = \otimes_{v\mid p}V_v$
by letting $U_v$ act by $\sigma_v$ for $v\mid p$
and letting $U_v$ act trivially for $v \nmid p$.
Let $\eta$ be a continuous character $(\A_F^f)^{\times}/F^{\times} \to A^{\times}$ such that 
for all $v$, the restriction of $\sigma$ and of $\eta$ to $U_v \cap
\O_{F_v}^{\times}$ coincide (such a character does not necessarily exist).

Let $S_{\sigma,\eta}(U,A)$ be the set of continuous functions
$f : B^{\times}\backslash (B\otimes_F \A_F^f)^{\times} \to V$ such that:
\begin{itemize}
\item
for all $g \in (B\otimes_F \A_F^f)^{\times}$ and $u \in U$, $f(gu) =
\sigma(u)^{-1}f(g)$

\item
for all $g \in (B\otimes_F \A_F^f)^{\times}$ and $z \in (\A_F^f)^{\times}$, $f(gz) =
\eta(z)^{-1}f(g)$
\end{itemize}

We can extend the action of $U$ on $(\sigma,V)$ to an action of
$U(\A_F^f)^\times$: we let $(\A_F^f)^{\times}$ act via $\eta$. We say that $U$ is small
enough (see for example \cite{Kis09a}, Paragraph 2.1.1) 
if for all $t\in (B\otimes_F\A_F^f)^{\times}$, $(U(\A_F^f)\cap
t^{-1}D^{\times}t)/F^{\times} = 1$. In this case, the functor $(\sigma,V) \mapsto
S_{\sigma,\eta}(U,A)$ is exact in $(\sigma,V)$.
In the following we will always assume that $U$ is small enough.

Let now $(\tilde{\sigma}_{v_0},V_{v_0})$ be a representation of $\G$
with coefficients in  
$A$, and for $v\in \Sigma_p'$, let $(\sigma_v,V_v)$ be a
representation of $U_v \simeq \O_D^{\times}$ as before. Let
$\tilde{\sigma}$ be the representation
$\tilde{\sigma}_{v_0}\otimes(\otimes_{v\in\Sigma_p'}\sigma_v)$ 
of $\G \times(\prod_{v\in \Sigma_p'}U_v)$ on
$\otimes_{v\in \Sigma_p}V_v$. Let $\sigma_{v_0}$ be the restriction of
$\tilde{\sigma}_{v_0}$ to $U_{v_0} = \O_D^{\times}$, we define as before
$\sigma$ a representation of $U$ on $\otimes_{v\in \Sigma_p}V_v$, and we
suppose that the character $\eta$ exists. We define a space of modular
forms $S_{\tilde{\sigma},\eta}(U,A)$ by setting
$S_{\tilde{\sigma},\eta}(U,A) = S_{\sigma,\eta}(U,A)$. We will endow the
space $S_{\tilde{\sigma},\eta}(U,A)$ with an additional structure (a
Hecke operator at $v_0$) in Paragraph \ref{heckep}.

\subsection{Hecke algebra}
\label{heckealg}

The group $(B\otimes_F\A_F^f)^{\times}$ acts on the set of functions on
$(B\otimes_F\A_F^f)^{\times}$ by $g\cdot f(z) = f(zg)$. 

Let $S$ be a finite set of places of $F$ containing all places above $p$
and all $v$ such that $U_v$ is not $\O_{B_v}^{\times}$, and $S'\subset S$ the set of
places $w$ such that $U_w$ is not $\O_{B_w}^{\times}$.  Let $T_S =
\Z[T_v,S_v,U_{\varpi_{w}}]_{v\not\in S,w\in S'}$ be a polynomial ring. We define an action
of $T_S$ on $S_{\sigma,\eta}(U,A)$ by:
\begin{itemize}
\item $T_v$ is the action of $U_v\smat{\varpi_v}001 U_v$.
\item $S_v$ is the action of $U_v\smat{\varpi_v}00{\varpi_v} U_v$.
\item $U_{\varpi_{w}}$ is the action of $U_w\smat{\varpi_w}001 U_w$.
\end{itemize} 
where $\varpi_v$ is a uniformizer of $F_v$. The actions of these
operators commute, and the definition of the Hecke
operators $T_v$ and $S_v$ does not depend on the choice of $\varpi_v$.

Let $\T_{\sigma,\eta}(U,A)$ be the $A$-algebra generated by the image of
$T_S$ in the ring of endomorphisms of $S_{\sigma,\eta}(U,A)$. 

\subsection{Hecke operators at places above $p$}

\subsubsection{Hecke operators}
\label{heckep}

We fix a representation $\tilde{\sigma}_{v_0}$ of $\G$, and
representations $\sigma_v$ of $U_v$ for $v\in \Sigma_p'$ as in 
Paragraph \ref{modular}. Consider the space of modular forms 
$S_{\tilde{\sigma},\eta}(U,A)$ of Paragraph \ref{modular}. 

We define an operator $W_{v_0}$ acting on
$S_{\tilde{\sigma},\eta}(U,A)$ by $(W_{v_0}f)(g) =
\tilde{\sigma}_{v_0}(\iota)f(g\varpi_{D,v_0})$
where $\varpi_{D,v_0}$ is the element of $(B\otimes_F\A_F^f)^{\times}$ that is
equal to $\varpi_D$ at
$v_0$ and $1$ everywhere else. 
One checks easily that $W_{v_0}f$ is indeed an element of
$S_{\tilde{\sigma},\eta}(U,A)$ if $f$ is.
Note that $W_{v_0}^2$ is multiplication by
$\eta(\varpi_{K,v_0})^{-1}$, where $\varpi_{K,v_0} = \varpi_{D,v_0}^2$.
It is clear from the definition that $W_{v_0}$ commutes with the action of
the Hecke algebra $\T_{\sigma,\eta}(U)$.

Suppose that $A$ contains a square root $\alpha$ of
$\eta(\varpi_{K,v_0})^{-1}$.
Then we get a decomposition $S_{\tilde{\sigma},\eta}(U) =
S_{\tilde{\sigma},\eta}(U,A)^{+}\oplus
S_{\tilde{\sigma},\eta}(U,A)^{-}$, where
$S_{\tilde{\sigma},\eta}(U,A)^{\pm}$ denotes the subspace of
$S_{\tilde{\sigma},\eta}(U,A)$
where $W_{v_0}$ acts by $\pm \alpha$.
If we replace $\sigma_{v_0}$ by $\xi\sigma_{v_0}$ without changing $\alpha$,
the space of modular forms $S_{\tilde{\sigma},\eta}(U,A)$ is unchanged,
$W_{v_0}$ is replaced by $-W_{v_0}$ and the $+$ and $-$ subspaces are
exchanged.

\subsubsection{The case of type $\red$}
\label{red}

Consider the following special case: let $(\sigma_{v_0},V_{v_0})$ be a
representation of $U_{v_0} \simeq \O_D^{\times}$ over $A$, and $\sigma_{v_0}'$ the
representation on $V_{v_0}$
defined by $\sigma_{v_0}'(g) = \sigma_{v_0}(\varpi_Dg\varpi_D^{-1})$. We can define a
representation $(\tilde{\sigma}_{v_0},\tilde{V}_{v_0})$ of $\G$ by
$\tilde{V}_{v_0}
= V_{v_0} \oplus V_{v_0}$, $U_{v_0}$ acts by $(\sigma_{v_0},\sigma_{v_0}')$ and $\iota$ acts
by $\tilde{\sigma}_{v_0}(x,y) = (y,x)$. Fix representations at places
$v\in \Sigma_p'$, a character $\eta$ and representations $\sigma$ and
$\tilde{\sigma}$ as in Paragraph \ref{modular}. 

Let  $\alpha$ be a square root of $\eta(\varpi_{K,v_0})^{-1}$.
We have two embeddings $i_+,i_- : S_{\sigma,\eta}(U,A) \to
S_{\tilde{\sigma},\eta}(U,A)$ given by $i_{\pm}(f)(g) =
(f(g),\pm\alpha^{-1}\tilde{\sigma}_{v_0}(\iota)f(g\varpi_{D,v_0}))$. 
The image of $i_{\pm}$ is 
$S_{\tilde{\sigma},\eta}(U,A)^{\pm}$ and $i_++i_-$ is a isomorphism
from $S_{\sigma,\eta}(U,A)^2$ to $S_{\tilde{\sigma},\eta}(U,A)$.

We will make use of this remark in the following situation:
$\tilde{\sigma}_{v_0}$ is
of the form $\sigmaG(\t) \otimes\sigma_{alg}$ for $\sigma_{alg}$ the restriction to $\G$ 
of some algebraic representation of $\GL_2$ (by an embedding as in Paragraph \ref{realization}), 
$\t$ is an inertial type of the form
$\red$ and $\sigmaG(\t)$ is the $\G$-representation attached to $\t$ in
Paragraph \ref{GpiK}. 

\subsection{Galois representations attached to quaternionic modular forms}

\subsubsection{General results}
\label{galoismodform}

Suppose now that $A$ is a $p$-adic field $E$ containing the unramified
quadratic extension $K'$ of $K$ and a square root $\sqrt{\varpi_K}$ of
$\varpi_K$. Then there is an embedding $\tilde{u}$ of $\G$ into $\GL_2(E)$ as in
Paragraph \ref{realization}.
Suppose that for all $v \mid p$, the representation $\sigma_v$ of Paragraph
\ref{modular} is of the form $\sigma_{v,alg} \otimes \sigma_{v,sm}$, where
$\sigma_{v,sm}$ is a smooth representation of $U_v$, and 
$\sigma_{v,alg}$ is the restriction to $U_v$ of an algebraic
representation of $\GL_2$ via $\tilde{u}|_{\O_D^{\times}}$.
We always assume that either $K=\Q_p$
or $\sigma_{v,alg}$ is trivial for all $v$.
If $K=\Q_p$, $\sigma_{v,alg}$ is the restriction of a representation
of the form $\Sym^{n_v}E^2\otimes\det^{m_v}$
and $k = n_v+2m_v+1$ does not depend on $v$.

We recall the construction and properties of Galois representations
associated to eigenforms in 
$S_{\sigma,\eta}(U,E)$. See for example \cite[Paragraph 3.1.14]{Kis09b}
for the link between these spaces of modular forms
and the classical spaces of automorphic representations, from which we
deduce the properties of the Galois representations attached to them.
Choose embeddings $i_p$, $i_{\infty}$ of $E$ into $\bar{\Q}_p$ and
$\C$ respectively. 

Let $\sigma_{p,alg} = \otimes_v\sigma_{v,alg}$ and $\sigma_{p,sm} =
\otimes_v\sigma_{v,sm}$. Let $\sigma_{\C,alg} = \sigma_{alg}\otimes_E\C$
and $\sigma_{\C,sm} = \sigma_{sm}\otimes_E\C$ and $\sigma_{\C} =
\sigma_{\C,alg}\otimes\sigma_{\C,sm}$, acting on the space $W_{\C}$. Then 
$\sigma_{\C,alg}$ can be viewed as a representation of
$B_{\infty}^{\times} = (B\otimes_{\Q}\mathbb{R})^{\times}$, and
$\sigma_{\C,sm}$ is a smooth representation of $U_p = \otimes_v U_v$. Let
$U'_p$ be a compact open subgroup of $U_p$ contained in $\ker
\sigma_{p,sm}$, and $U'$ the compact subgroup of
$(B\otimes_F\A_F)^{\times}$ which is the same as $U$ but with $U_p$
replaced by $U'_p$.

Let $C^{\infty}(B^{\times}\backslash (B\otimes_{F}\A_{F})^{\times}/U')$ the space of 
smooth functions with values in $\C$. It is endowed with a right action of
$B_{\infty}^{\times}$.

We denote by $(\sigma^{\vee},W^{\vee})$ the dual of a 
representation $(\sigma,W)$.
Let $\phi : S_{\sigma,\eta}(U,E) \to
\Hom_{B_{\infty}^{\times}}(W_{\C}^{\vee},C^{\infty}(B^{\times}\backslash
(B\otimes_{F}\A_{F})^{\times}/U'))$
be the map defined by $\phi(f) = w \mapsto (x \mapsto
w(\sigma_{\C,alg}(x_{\infty})^{-1}\sigma_{p,alg}(x_p)f(x^{\infty})))$,
where $x = (x^{\infty},x_{\infty}) \in (B\otimes_F\A_F^f)^{\times}\times
B_{\infty}^{\times}$.  We denote by $\phi_w(f) \in
C^{\infty}(B^{\times}\backslash (B\otimes_{F}\A_{F})^{\times}/U')$ the
element $\phi(f)(w)$ for $w \in W_{\C}^{\vee}$.

Let $\pi = \otimes_v\pi_v$ be the irreducible automorphic representation
of $B^{\times}$ generated by some non-zero
$\phi_w(f)$ for $w \in W_{\C}^{\vee}$ and $f \in S_{\sigma,\eta}(U,E)$
that is an eigenform for $T_S$.  Then $\pi_{\infty}$ is isomorphic to
$W^{\vee}_{alg,\C}$, and has central character $\eta_{\C}(z) =
N_{F/\Q}(z_{\infty})^{1-k}N_{F/\Q}(z_p)^{k-1}\eta(z_p)^{-1}$.
Let $\rho_f : G_F  \to \GL_2(\bar{\Q}_p)$ be the Galois representation
attached to $\pi$, so that for all $v$ not in $S$, the characteristic
polynomial $\rho_f(\Frob_v)$ is $X^2-t_vX+N(v)s_v$, where $\Frob_v$ is an
arithmetic Frobenius at $v$, and $t_v$ and $s_v$ are the eigenvalues of
the Hecke operators $T_v$ and $S_v$ acting on $f$.
Then 
$\rho_f$ has determinant $\eps\eta$
and for all $v\mid p$, $\rho_f|_{G_{F_v}}$ is potentially semi-stable with
Hodge-Tate weights $(m_v,m_v+n_v+1)$ if $K=\Q_p$ and
$(0,1)_{\tau\in\Hom(K,\bar{\Q}_p)}$ otherwise, and 
$\WD(\rho_f|_{G_{F_v}})^{F-ss}$ is isomorphic to $\LL_D(\pi_v^{\vee})$.

\subsubsection{Structure at $p$}
\label{structurep}

Let $v$ be in $\Sigma_p$.
Let $\varphi : \sigma_{\C}^{\vee} \to \pi$ be given by $w \mapsto \phi_w(f)$,
and $\varphi_v : \sigma_{\C}^{\vee} \to \pi_v$ be the projection to $\pi_v$.
It is a non-zero $U_v$-equivariant morphism (where $U_v$ acts on
$\sigma_{\C}^{\vee}$ via its action by $\sigma_{v,sm}$), 
hence $\pi_v|_{U_v}$ contains some
irreducible constituent of $\sigma_{v,sm}^{\vee}$. In particular if
$\sigma_{v,sm}$ is a copy of representations $\sigma_D(\t_v)$ attached to some discrete series
inertial type $\t_v$ as in Paragraph \ref{typeD},  $\rho_f|_{G_{F_v}}$
is of type $\t_v$ by Proposition \ref{typesD}.

Fix $v_0 \in \Sigma_p$
and suppose that $\sigma_{v_0,sm}$ is in fact a
representation of $\G$. By the embedding $\tilde{u}$, the
representation $\sigma_{v_0,alg}$ also extends to a representation of
$\G$. Suppose that $W_{v_0}f = \alpha f$ where $W_{v_0}$ is the Hecke operator
defined in Paragraph \ref{heckep}. Let $\mu$ be the central character of
$\sigma_{v_0,alg}$.
We extend the representation $\sigma_{v_0,sm}^{\vee}$ of $\G$ 
to a representation of $B_{v_0}^{\times} = D^{\times}$ by 
$\sigma_{v_0,sm}^{\vee}\otimes\unr(\alpha\mu(\sqrt{\varpi_K}))$ (here
$\sigma_{v_0,sm}$ is seen as a representation of $D^{\times}$ via the
canonical map $D^{\times} \to \G$).
Then $\varphi_{v_0}$ is equivariant for the action of the group $D^{\times}$
so that $\pi_{v_0}$ is isomorphic to
$\sigma_{v_0,sm}^{\vee}\otimes\unr(\alpha\mu(\sqrt{\varpi_K}))$ if
$\sigma_{v_0,sm}$ is irreducible.
This gives the following:

\begin{lemm}
\label{eigenvalueexttype}
If $\sigma_{v_0,sm}=\sigmaG(\t)$ 
for some discrete series inertial type $\t$, $\sigma_{v_0,alg}$ has central
character $\mu$ and if $W_{v_0}f = \alpha f$ then
$\rho_f|_{G_{F_{v_0}}}$ is of extended type
$\LL_D(\sigmaG(\t))\otimes\unr(\alpha\mu(\sqrt{\varpi_K}))^{-1}$.
\end{lemm}

\section{Proof of the main theorems}
\label{proof}

\subsection{Notation}
\label{notationproof}

In this section we fix $K$ and a continuous representation
$\bar{\rho} : G_K \to \GL_2(\F)$.

When $K = \Q_p$, we assume that 
$p \geq 5$ when $\bar{\rho}$ is a twist of an
extension of the trivial character by the cyclotomic character (we need
this condition to apply the results of \cite{Pas} and \cite{HT15}).

When $K \neq \Q_p$, we assume that $\bar{\rho}$ is not a twist of an
extension of the trivial character by the cyclotomic character, and
whenever a Hodge-Tate type $w$ appears we always mean $w = w_0$.

We fix $\varpi_K$ a uniformizer of $K$.

Let $\bar\psi$ be the character $\omega^{-1}\det\bar\rho$ of $G_K$, 
which we see also as a character
$K^\times$ via local class field theory. We fix
$\alpha\in\bar\F_p$ such that $\alpha^2 = \bar\psi(\varpi_K)^{-1}$.

For any irreducible representation $\gamma$ of $\Gamma = \Gamma_K$, we fix
an inertial type $\t_{\gamma}$ and a
representation $\sigmaG(\t_{\gamma})$ as in the proof of Proposition
\ref{liftistype}, 
a lift $\psi_{\gamma}$ of $\bar\psi$ that is compatible with
$\t_{\gamma}$ and $w_0$, 
and an 
extended type $\t_{\gamma}^+$ such that $\t^+_{\gamma}$ is compatible
with $(\t_{\gamma},\psi_{\gamma})$ and $\t_{\gamma}^+ =
\LL_D(\sigmaG(\t_{\gamma}))\otimes\unr(a_{\gamma})^{-1}$
for an $a_{\gamma}$ lifting $\alpha$.

\subsection{Definition of $\mu_{\bar\rho}$}
\label{defmu}

We are now able to define the linear form $\mu_{\bar\rho}$:
we define it to be the linear form on $\R(\Gamma)$ such that
$\mu_{\bar{\rho}}([\gamma]) =
e(R^{\square,\psi_{\gamma}}(w_0,\t_{\gamma}^{+},\bar{\rho})/\pi)$
for any irreducible representation $\gamma$ of $\Gamma$.
It is clear that 
$\mu_{\bar{\rho}}(\gamma) = 0$ if $i_{\bar\rho}(\gamma) = 0$.

We must now prove that $\mu_{\bar\rho}$ satisfies the properties claimed
in Theorems \ref{mainQp} and \ref{mainK}.

Let $\t$ be a discrete series inertial type, $w$ a Hodge-Tate type (with
$w = w_0$ if $K \neq \Q_p$), $\psi$ a lift of $\bar\psi$ that is
compatible with $\t$ and $w$, and
$\t^+$ an extended type compatible with $(\t,\psi)$.

If $R^{\square,\psi}(w,\t^{ds},\bar{\rho}) = 0$
then by the results of Paragraph \ref{reformulation}
we have that 
$\mu_{\bar\rho}([\bar\sigmaG(\t)\otimes \bar\sigma_w]) = 0$.
So we need only prove the equalities of Theorems \ref{mainQp} and
\ref{mainK} when
$R^{\square,\psi}(w,\t^{ds},\bar{\rho}) \neq 0$. 
This is the object of the rest of this Section.

\subsection{Global realization in characteristic $p$}
\label{residualrepr}

We start by realizing $\bar\rho$ in some global Galois representation.

\begin{prop}
\label{existsresrepr}
There exist a totally real field $F$ and a
continuous irreducible representation
$\bar{r} : G_F \to \GL_2(\f)$, such that:
\begin{enumerate}
\item
the number of places of $F$ above $p$ has the same parity as the number
of infinite places of $F$
\item
for all $v \mid p$, $F_v$ is isomorphic to $K$
\item
for all $v \mid p$, $\bar{r}|_{G_{F_v}} \simeq \bar{\rho}$
\item
$\bar{r}$ is unramified outside $p$
\item
$\bar{r}$ is totally odd
\item
$\bar{r}$ is modular
\item
the restriction of $\bar{r}$ to $G_{F(\zeta_p)}$ is absolutely
irreducible, and if $p=5$, $\bar{r}$ does not have projective image
isomorphic to $\text{PGL}_2(\F_5)$.
\end{enumerate}
\end{prop}

\begin{proof}
All conditions except the first follow from Corollary A.3 of
\cite{GK}. We can ensure that the first condition is satisfied by taking
an $F$ such that the number of places of $F$ above $p$ and the number of
infinite places of $F$ are both even. Indeed, note that the proof of
Corollary A.3 of \cite{GK} starts (in Proposition A.1) by considering an
auxiliary number field $E$ which is totally real and such that for all $v
\mid p$, $E_v$ is isomorphic to $K$, and the $F$ we get is a finite
extension of $E$. If we impose to $E$ the additional condition that
$2[K:\Q_p]$ divides $[E:\Q]$, then the parity condition will be
satisfied.
\end{proof}

\subsection{Global realizations in characteristic $0$}

\subsubsection{Global data}
\label{globdata}

From now on we fix a field $F$ and a representation $\bar{r}$ satisfying
the conditions of Proposition \ref{existsresrepr}.

Let $\Sigma_p$ be the set of places of $F$ above $p$. We fix a
$v_0 \in \Sigma_p$ and denote $\Sigma_p\setminus\{v_0\}$ by $\Sigma_p'$
as before.

Let $B$ be the quaternion algebra with center $F$ that is ramified
exactly at the infinite places of $F$ and at all places in $\Sigma_p$.
Such a $B$ exists thanks to condition (1) of Proposition
\ref{existsresrepr}.
Let $D$ be the non-split quaternion algebra over $K$. 
Let $\O_B$ be a maximal order in $B$.
For all $v$ not dividing $p$, we fix an isomorphism $\O_{B_v}^{\times} \simeq
\GL_2(\O_{F_v})$, and for all $v \in \Sigma_p$, we fix an isomorphism
$\O_{B_v}^{\times} \simeq \O_D^{\times}$.

Let $\varpi_D$ a uniformizer of
$D = B_{v_0}$ such that $\varpi_D^2 = \varpi_K$
where $\varpi_K$ is our fixed uniformizer of $K$.
We set $\G = \G_{\varpi_K}$.

We choose an auxiliary place $v_1 \nmid p$ such that $Nv_1 \neq 1 \pmod
p$, the ratio of the eigenvalues of $\bar{r}(\Frob_{v_1})$ is not
$Nv_1^{\pm 1}$, and the characteristic of $v_1$ is large enough so that
for any quadratic extension $F'$ of $F$ and any $\zeta$ a root of unity
in $F'$, $v_1 \nmid \zeta+\zeta^{-1}-2$. The existence of such a place
$v_1$ follows from \cite{DDT}, Lemma 4.11 and \cite{Kis09a}, Lemma 2.2.1.

We let $U$ be the compact open subgroup of $(B\otimes_F\A_F^f)^{\times}$ such that
$U_v = \O_D^{\times}$ for $v \in \Sigma_p$, $U_v = \GL_2(\O_{F_v})$ for $v
\not\in \Sigma_p$ and $v\neq v_1$, and finally $U_{v_1}$ is the set of elements
of $\GL_2(\O_{F_{v_1}})$ that are upper triangular unipotent modulo $v_1$. 
The last condition we imposed on $v_1$ ensures that $U$ is small enough in 
the sense of Paragraph \ref{modular} (see \cite{Kis09a}, Section 2.1.1.).

\subsubsection{Modular lift}
\label{modularlift}

We want know to show that the representation $\bar{r}$ can be lifted
to an appropriate modular Galois representation.

\begin{lemm}
\label{existsmodformirr} 
For all $v\in \Sigma_p$, let $\t_v$ be an inertial type such that
$\gamma_v = \sigma_D(\t_v)$ is an irreducible representation of $\ell^{\times}$, and
$\psi_v$ be a character of $G_{F_v}$. Suppose that the ring
$R_{p} = 
\hat{\otimes}_{v\in \Sigma_p}R^{\square,\psi_v}(w_0,\t_v^{ds},\bar{\rho})$
is not zero.
Let $\sigma = \otimes_{v\in \Sigma_p} \sigma_D(\t_v)$.
Then there exists $\eta$ satisfying the compatibility conditions with $\sigma$
of Paragraph \ref{modular}, and which restricts to $\psi_v$ on
$F_{v}^{\times}$ for all $v \in \Sigma_p$, and 
the space of modular forms $S_{\sigma,\eta}(U,\O)$
contains an eigenform $f$
whose associated Galois representation
has its reduction modulo $p$ isomorphic to $\bar{r}$.
\end{lemm}

\begin{proof}
For the existence and construction of $\eta$, 
see Paragraph 5.4.1. of \cite{GK}.

We now prove the existence of the eigenform $f$.

Suppose first that none of the $\t_v$ is scalar.
Then each $R^{\square,\psi_v}(w_0,\t_v^{ds},\bar{\rho})$
parametrizes only potentially crystalline representations.
Corollary 3.1.7 of \cite{G} applied to our situation gives that
for each irreducible component of $\spec R_{p}[1/p]$, there exists
a lift $r$ of $\bar{r}$ that is modular, unramified outside $p$, with
determinant $\eta\eps$, potentially crystalline with Hodge-Tate
weights $(0,1)$ at each $v\in \Sigma_p$ and for each $F_v \to \bar{\Q}_p$ 
and defining a point on the given irreducible component. 
The representation $r$ comes from some automorphic form on a quaternion
algebra over $F$. Thanks to the local conditions on $r$, 
we can suppose that $r$ comes from a modular form $f$
on $B$, and that $f \in S_{\sigma,\eta}(U,\O)$.
This proves the claim in this case, and in particular whenever
$\bar{\rho}$ is not isomorphic to a twist of $\mat {\omega}*01$ by
Proposition \ref{dim1}.

When $\bar{\rho}$ is isomorphic to a twist of $\mat {\omega}*01$
(in the case $K=\Q_p$), Corollary 3.1.7 of \cite{G} is not enough: for 
$\dim\gamma_v = 1$, we need points on 
$\spec R^{\square,\psi_v}(w_0,\t_{v}^{ds},\bar{\rho})[1/p]$ 
corresponding to potentially
semi-stable representations that are not potentially crystalline.
We make use of \cite[Théorème 3.2.2]{BD} instead: it allows us to
choose $r$ such that each $r|_{G_{F_v}}$ is not potentially
crystalline when $\dim\gamma_v=1$. Note that in this situation,
each $R^{\square,\psi_v}(w_0,\t_{\gamma_v}^{ds},\bar{\rho})$ 
is irreducible.
\end{proof}

Let $\t$ be a discrete series 
inertial type and $w$ be a Hodge-Tate type, and let $\psi$ be the
character defined in Paragraph \ref{notationproof}.

\begin{prop}
\label{existsmodform}
Suppose that $R^{\square,\psi}(w,\t^{ds},\bar{\rho}) \neq 0$.
There exists a character $\eta$ of
$(\A_F^f)^{\times}/F^{\times}$ which restricts to $\psi$ on
$F_{v}^{\times}$ for all $v \in \Sigma_p$
such that for $\sigma = 
\otimes_{v\in \Sigma_p} (\sigma_D(\t)\otimes
\sigma_{w})$, the space of modular forms $S_{\sigma,\eta}(U,\O)$
contains an eigenform $f$
whose associated Galois representation
has its reduction modulo $p$ isomorphic to $\bar{r}$.
\end{prop}

\begin{proof}
As before, the existence of 
$\eta$ satisfying the compatibility conditions with
$\sigma$ of Paragraph \ref{modular} and whose
restriction to 
$F_{v}^{\times}$ 
coincides $\psi$ for all $v\in \Sigma_p$,
comes from Paragraph 5.4.1. of \cite{GK}.

If $R^{\square,\psi}(w,\t^{ds},\bar{\rho}) \neq 0$,
by Theorems \ref{BMquatQp} and \ref{BMquatK} there exist an
irreducible representation $\gamma$ of $\prod_{v\in
\Sigma_p} \ell^{\times}$, $\gamma = \otimes_{v\in
\Sigma_p}\gamma_v$ appearing as an irreducible constituent of
$\otimes_{v\in \Sigma_p}(\bar{\sigma_D}(\t)\otimes\bar{\sigma_{w}})$ and
characters $\psi_{v}$ that are equal to a finite order character times
the cyclotomic character such that $\hat{\otimes}_{v\in
\Sigma_p}R^{\square,\psi_v}(w_0,\t_{v}^{ds},\bar{\rho}) \neq 0$,
where $\t_{v}$ is the inertial type attached to the representation
$\gamma_v$ as in the proof of Proposition \ref{liftistype}.

We can apply Lemma \ref{existsmodformirr} to the family of types $(\t_v)$,
as by construction $\sigma_D(\t_v) = \gamma_v$ is an irreducible
representation of $\ell^\times$. Then the result follows from 
Lemma 2.1 of \cite{GS}.
\end{proof}

In particular in the conditions of Proposition \ref{existsmodform}, 
${r_f}|_{G_{F_{v_0}}}$ has determinant
$\eps\psi$, inertial type $\t$ and Hodge-Tate type $w$ where $r_f$ is the
Galois representation attached to $f$ as in Paragraph
\ref{galoismodform}.

\subsection{Patching}
\label{patching}

Let $(\t_v,w_v)_{v\in \Sigma_p}$ be a family of discrete series inertial
types and Hodge-Tate types, with $\t_{v_0} = \t$ and $w_{v_0} = w$.
Let also $(\psi_v)_{v\in\Sigma_p}$ be a family of characters of $G_K =
G_{F_v}$.

We suppose in this paragraph that
there exists a character $\eta$ of
$(\A_F^f)^{\times}/F^{\times}$
that restricts to $\psi_v$ on $F_v^\times$ for all $v\in\Sigma_p$
and such that for $\sigma = 
\otimes_{v\in \Sigma_p} (\sigma_D(\t_v)\otimes
\sigma_{w_v})$, the space of modular forms $S_{\sigma,\eta}(U,\O)$
contains an eigenform $f$
whose associated Galois representation
has its reduction modulo $p$ isomorphic to $\bar{r}$.

Consider now $\tilde{\sigma} = 
(\sigmaG(\t))\otimes\sigma_w)\otimes
(\otimes_{v\in \Sigma_p'}(\sigma_D(\t_v)\otimes \sigma_{w_v})$. 
It is a representation of $\G\times\prod_{v\in \Sigma'_p}U_v$ where $U_v
= \O_D^{\times}$.
The space of modular forms $S_{\tilde{\sigma},\eta}(U,\O)$ is either the
same as $S_{\sigma,\eta}(U,\O)$ as an $\O$-module and a $T_S$-algebra 
(if $\t$ is of the form
$\scal$ or $\irr$, as in this case $\sigmaG(\t)$ acts on the same
space as $\sigma_D(\t)$) 
or is two copies of $S_{\sigma,\eta}(U,\O)$ (if $\t$
is of the form $\red$, as in this case $\sigmaG(\t)$ acts on a
space which is two copies $\sigma_D(\t)$, see
Paragraph \ref{red}). In any case, it contains a copy of the form 
$f$ given by Proposition \ref{existsmodform}.
Let $\m$ be a maximal ideal of $T_S$ 
containing $p$ such that $f$ is in $S_{\tilde{\sigma},\eta}(U,\O)_{\m}$.

Let $\T_{\tilde\sigma,\eta}(U,\O)_{\m}$ be the image of 
$T_{S,\m}$ in the endomorphism ring of $S_{\tilde{\sigma},\eta}(U,\O)_{\m}$.
Any eigenform in $S_{\tilde{\sigma},\eta}(U,\O)_{\m}$ has an associated
Galois representation with residual representation isomorphic to
$\bar{r}$, which is absolutely irreducible.
By the main result of \cite{Taya} and the Jacquet-Langlands
correspondence (see \cite[Lemma 1.3]{Tayb}) we deduce that
we have a Galois representation
$r_{\m} : G_F \to \GL_2(\T_{\tilde\sigma,\eta}(U,\O)_{\m})$ 
coming from all the eigenforms in 
$S_{\tilde{\sigma},\eta}(U,\O)_{\m}$. In particular, $r_{\m}$ is
unramified outside $p$.

Let $\R$ be the universal
deformation ring for deformations of $\bar{r}$ that are
unramified outside $\Sigma_p$ and $\R^{\square}$ the framed analogue.
Then we have a map $\R \to \T_{\tilde\sigma,\eta}(U,\O)_{\m}$ coming from $r_{\m}$, so
$\R$ acts on $S_{\tilde{\sigma},\eta}(U,\O)_{\m}$ via the Hecke
operators.

Let $R^{\square}(\bar{\rho})$ the ring classifying lifts of
$\bar{\rho}$, and $\R_p^{\square} = \hat{\otimes}_{v\in
\Sigma_p}R^{\square}(\bar{\rho})$.
Let $R_p$ be 
$\hat{\otimes}_{v\in \Sigma_p}R^{\square,\psi_v}(w_v,\t_v^{ds},\bar{\rho})$
seen as a quotient of $\R^{\square}_p$, and 
$\R' = \R^{\square}\otimes_{\R^{\square}_p}R_p$.
The ring $\R'$ is the universal ring for lifts of $\bar{r}$ that are
unramified outside $p$ and potentially semi-stable of inertial type
$\t_v$, of some discrete series extended type, of Hodge-Tate type $w_v$ and
determinant $\eps\psi$ at each place $v\in \Sigma_p$.
Let $M_0 = \R^{\square}\otimes_{\R}S_{\tilde{\sigma},\eta}(U,\O)_{\m}$. Then the
action of $\R^{\square}$ on $M_0$ factors through $\R'$ by the results
recalled in Paragraph \ref{structurep}.

We decompose the reduction $\bar{\sigma}$ of $\tilde{\sigma}$ modulo $p$ 
as a direct sum of representations of
$\Gamma\times\prod_{v\in \Sigma'_p}\ell^{\times}$: 
$\bar{\sigma} = \oplus_{\gamma}\gamma^{\oplus n_{\gamma}}$,
with $\gamma = \gamma_{v_0}\otimes(\otimes_{v\in \Sigma_p'}\gamma_v)$.
This gives a decomposition 
$\bar{\sigma} = \oplus_{\gamma}\gamma^{n_{\gamma}}$
as a representation
of $\G\times \prod_{v\in \Sigma_p'}U_v$.

Using the techniques of \cite{Kis09a}, Section (2.2) 
we construct the following objects:
\begin{enumerate}
\item 
a ring $R_{\infty}$ which is a power series ring on $R_p$ (this is
$\bar{R}_{\infty}$ of \cite{Kis09a}).

\item 
a ring $S_{\infty}$ which is a power series ring on $\O$ (this is
$\O[[\Delta_{\infty}]]$ of \cite{Kis09a}).

\item
an $(S_{\infty},R_{\infty})$-module $M_{\infty}$ that is free as an
$S_{\infty}$-module, and such that $M_0$ is a quotient of $M_{\infty}$.

\item
a $(S_{\infty},R_{\infty})$-linear operator $W_{v_0}$ on $M_{\infty}$
compatible with the Hecke operator $W_{v_0}$ on $M_0$ defined in Paragraph
\ref{heckep}, with $W_{v_0}^2 = \psi_{v_0}(\varpi_{K})^{-1}$.

\item
a decomposition $M_{\infty}\otimes\F =
\oplus_{\gamma}\bar{M}_{\infty,\gamma}^{\oplus n_{\gamma}}$ as
$(S_{\infty},R_{\infty})$-module, such that each
$\bar{M}_{\infty,\gamma}$ is a finite free $S_{\infty}\otimes\F$-module,
and such that moreover $\bar{M}_{\infty,\gamma}$ does not depend on
$(\t_{v_0},w_{v_0})$ or the $(\t_v,w_v)$, in a sense that is explained
below.
\end{enumerate}

The only part which is not a copy of the arguments of \cite{Kis09a} is (4). 
The module $M_{\infty}$ is built by patching
modules $M_n = \R^{\square}\otimes_{\R}S_{\tilde{\sigma},\eta}(U_n,\O)_{\m_n}$
for some choice of compact open subgroup $U_n$ which is maximal at $v$
for all $v\mid p$, and for some choice of maximal ideal $\m_n$ (with
$U_0=U$ and $\m_0=\m$). In
particular, for each $n$ there is a Hecke operator $W_{v_0}$ on $M_n$
as in Paragraph \ref{heckep} with $W_{v_0}^2 = \eta(\varpi_{K,v_0})^{-1}
= \psi_{v_0}(\varpi_K)^{-1}$
which is compatible with the surjective map $M_n \to M_0$. 
Moreover, the action of $W_{v_0}$ commutes with the right action of the
subgroup of $(B\otimes_F\A_F^f)^{\times}$ of elements that are trivial at $v_0$, hence
$W_{v_0}$ is $R_{\infty}$- and $S_{\infty}$-linear.
Once a square root of $\eta(\varpi_{K,v_0})^{-1}$ is fixed, then for each $n$
we have a decomposition $M_n = M_n^+\oplus M_n^-$ in
sub-$(S_{\infty},R_{\infty})$-modules according to the
eigenvalues of $W_{v_0}$, and this decomposition is compatible to the
decomposition $M_0^+ \oplus M_0^-$. We apply the patching argument not
only to $M_n$, but to the decomposition $M_n^+\oplus M_n^-$, which gives
a decomposition $M_{\infty} = M_{\infty}^+ \oplus M_{\infty}^-$ and an 
operator $W_{v_0}$ on $M_{\infty}$ with the required properties. 
Note that $M_{\infty}^+$ and $M_{\infty}^-$ are also free as
$S_{\infty}$-modules.

\smallskip

Let $\gamma$ be an irreducible smooth representation of $\G
\times\prod_{v\in \Sigma'_p}U_v$ in characteristic $p$, and 
$\tilde{\gamma}$ a smooth lift of $\gamma$ (as in Proposition
\ref{liftistype} for the part which is a representation of $\G$ and
by Teichmüller lift for the parts which are representations of $U_v =
\O_D^\times$).  By Lemma
\ref{existsmodformirr}, there exists a character $\eta_{\gamma}$ of 
$(\A_F^f)^{\times}/F^{\times}$ lifting $\bar\eta$ and characters
$(\psi_{\gamma,v})_{v\in\Sigma_p}$ satisfying the conditions at the
beginning of this Paragraph, so we can make the constructions with the
space of modular forms $S_{{\tilde{\gamma}},\eta_{\gamma}}(U,\O)_{\m}$.
We denote by $M_{\infty,\gamma}$ the patched module we obtain
(although it depends not only on $\gamma$ but also on other data).
Then (5)
means that $\bar{M}_{\infty,\gamma}$ is isomorphic to the reduction
modulo $p$ of $M_{\infty,\gamma}$.
We also have a Hecke operator $W_{v_0}$ on $M_{\infty,\gamma}$ 
and a decomposition 
$M_{\infty,\gamma} = M_{\infty,\gamma}^+ \oplus M_{\infty,\gamma}^-$ 
that reduces to $\bar{M}_{\infty,\gamma} =
\bar{M}_{\infty,\gamma}^+ \oplus \bar{M}_{\infty,\gamma}^-$.

\subsection{Equality of multiplicities}

Recall that $d_{\t} = 2$ if $\t$ is of the form $\red$ and $d_{\t} = 1$
otherwise.
As in Lemma (2.2.11) of \cite{Kis09a}, $M_{\infty}$ has rank $0$ or
$d_{\t}$ at each generic point of $R_{\infty}$ and $e(M_{\infty}/\pi
M_{\infty},R_{\infty}/\pi R_{\infty}) \leq d_{\t}e(R_{\infty}/\pi R_{\infty})$
with equality if and only if the
support of $M_{\infty}$ is all of $\spec R_{\infty}$ (we already know that
it is a union of irreducible components of $\spec R_{\infty}$).

Our main ingredient is the following Proposition, which we
prove using the results of Paragraph \ref{reformulation}
and the methods of Section (2.2) of \cite{Kis09a}:

\begin{prop}
\label{equality}
$e(M_{\infty}/\pi M_{\infty},R_{\infty}/\pi R_{\infty})
= d_{\t}e(R_{\infty}/\pi R_{\infty})$
\end{prop}

\begin{proof}
Suppose first that $w_v = w_0$ for all $v \in \Sigma_p$.
By reasoning as in the proof of Proposition \ref{existsmodform}, we see 
that the support of the module $M_0$ meets each irreducible component
of $\spec R_{\infty}[1/p]$.
The irreducible components of $\spec R_{\infty}[1/p]$ are
connected components, hence we can
apply the criterion of Lemma 4.3.8 of \cite{GK}: the equality of Proposition \ref{equality}
holds if and only if the support of $M_{0}$ meets every irreducible component of
$\spec R_{\infty}[1/p]$.

\smallskip

Return now to the case without conditions on the $w_v$ (in the case $K
= \Q_p$).
It follows from the decomposition $M_{\infty}\otimes\F =
\oplus_{\gamma}\bar{M}_{\infty,\gamma}^{\oplus n_{\gamma}}$ that
$e(M_{\infty}/\pi,R_{\infty}/\pi)
= \sum_{\gamma}n_{\gamma}e(\bar{M}_{\infty,\gamma},R_{\infty}/\pi)$.
Moreover $e(R_{\infty}/\pi R_{\infty}) = \prod_{v\in
\Sigma_p}e(R^{\square,\psi_v}(w_v,\t_{\gamma_v}^{ds},\bar{\rho})/\pi)$.
As $K = \Q_p$, it follows from the theorems of Paragraph \ref{reformulation} that
we have $d_{\t}e(R_{\infty}/\pi R_{\infty}) =
\sum_{\gamma}n_{\gamma}(\dim\gamma_{v_0})
\prod_{v\in\Sigma_p}e(R^{\square,\psi}(w_0,\t_{\gamma_v}^{ds},\bar{\rho})/\pi)$.

We denote by $R_{\infty,\gamma}$ the ring that is 
the analogue of $R_{\infty}$ but with
$(w_0,\t_{\gamma_v})$ instead of $(w_v,\t_v)$ for all $v \in \Sigma_p$,
and the characters $(\psi_{\gamma,v})_{v\in\Sigma_p}$ 
defined at the end of Paragraph \ref{patching}
instead of $(\psi_v)$.  Let also $R_{p,\gamma}$ be the analogue of $R_p$.

Then the image of $R_{\infty}$
and $R_{\infty,\gamma}$ in the
endomorphisms of $\bar{M}_{\infty,\gamma}$ is the same, as follows from (4) of
the properties of patching, hence $e(\bar{M}_{\infty,\gamma},R_{\infty}/\pi)
= e(\bar{M}_{\infty,\gamma},R_{\infty,\gamma}/\pi)$.
Moreover $e(R_{\infty,\gamma}/\pi) =
\prod_{v\in\Sigma_p}e(R^{\square,\psi_{\gamma,v}}(w_0,\t_{\gamma_v}^{ds},\bar{\rho})/\pi)$, 
and we have 
$e(\bar{M}_{\infty,\gamma},R_{\infty,\gamma}/\pi) = (\dim\gamma_{v_0})
e(R_{\infty,\gamma}/\pi)$ by applying the part of Proposition \ref{equality} 
that we have already proved to $\bar{M}_{\infty,\gamma}$
and $R_{\infty,\gamma}$, which concludes the proof of Proposition \ref{equality} in the
general case.
\end{proof}

\subsection{Action of the Hecke operator at $p$}
\label{actionhecke}

We apply the results of the preceding paragraphs in the situation coming
from Proposition \ref{existsmodform}, so in particular $w_v = w$ and
$\t_v = \t$ and $\psi_v = \psi$ for all $v\in\Sigma_p$.
Recall that we have chosen a square root $\alpha$ of
$\bar{\eta}(\varpi_{K,v_0})^{-1} = \bar{\psi}(\varpi_K)^{-1}$.

Let $(\t^+,\t^-) = (\t_{v_0}^+,\t_{v_0}^-)$ be
a pair of conjugate extended types compatible
to $(\t,\psi)$. 
We set 
$R^+_{\infty} = 
R^{\square,\psi}(w_{v_0},\t_{v_0}^+,\bar{\rho})
\otimes_{R^{\square,\psi}(w_{v_0},\t_{v_0}^{ds},\bar{\rho})}R_{\infty}$ and 
$R^-_{\infty} = 
R^{\square,\psi}(w_{v_0},\t_{v_0}^-,\bar{\rho})
\otimes_{R^{\square,\psi}(w_{v_0},\t_{v_0}^{ds},\bar{\rho})}R_{\infty}$ (so that
$R_{\infty} = R_{\infty}^+ = R_{\infty}^-$ when $\t$ is of the form
$\red$, and these rings differ only when $\t$ is of the form $\scal$ or
$\irr$).

We make a choice for $\sigmaG(\t)$ so that 
$\t^+ = \LL_D(\sigmaG(\t))\otimes\unr(a\sqrt{\varpi_K}^{|w|})^{-1}$ for
$a$ a lift of $\alpha$ with $a^2 = \psi(\varpi_{K})^{-1}$ (recall that
$|w| = n+2m$ if $K=\Q_p$ and $w = (n,m)$, and set $|w_0|=0$ for any $K$).

We can consider the Hecke operator $W_{v_0}$ acting on all the spaces of
modular forms that we have defined. This gives a decomposition
$M_{\infty} = M_{\infty}^+ \oplus M_{\infty}^-$ as in Paragraph
\ref{patching}, where $M_{\infty}^+$ is the submodule on which $W_{v_0}$
acts by a lift of $\alpha$, and decompositions $M_n = M_n^+\oplus M_n^-$
for the modules $M_n$.

The action of $R_{\infty}$ on
$M_{\infty}^+$ factors through $R_{\infty}^+$, and similarly for
$R_{\infty}^-$. Indeed this is true for each $M_n = M_n^+\oplus M_n^-$
by Lemma \ref{eigenvalueexttype}.

We can do the same thing for each irreducible representation $\gamma$
of $\Gamma$: recall that for each $\gamma$
we made a choice in Paragraph \ref{notationproof} of 
an inertial type $\t_{\gamma}$ and a representation
$\sigmaG(\t_{\gamma})$ of $\G$
such that $\bar{\sigmaG}(\t_{\gamma})$ is isomorphic to $\gamma$ and an
extended type
$\t_{\gamma}^+ =
\LL_D(\sigmaG(\t_{\gamma}))\otimes\unr(a_{\gamma})^{-1}$
for an $a_{\gamma}$ lifting $\alpha$ with $W_{v_0}^2 = a_{\gamma}^2$ on
$M_{\infty,\gamma}$.
Let
$\t_{\gamma}^- =
\LL_D(\xi\sigmaG(\t_{\gamma}))\otimes\unr(a_{\gamma})^{-1} =
\LL_D(\sigmaG(\t_{\xi\gamma}))\otimes\unr(a_{\gamma})^{-1} =
\LL_D(\sigmaG(\t_{\gamma}))\otimes\unr(-a_{\gamma})^{-1}$.

For $\gamma$ of dimension $2$ we set $R_{\infty,\gamma}^+ =
R_{\infty,\gamma}^- = R_{\infty,\gamma}$ and for $\gamma$ of dimension $1$
we set $R^+_{\infty,\gamma} = 
R^{\square,\psi}(w_{0},\t_{\gamma}^+,\bar{\rho})
\otimes_{R^{\square,\psi}(w_{0},\t_{\gamma}^{ds},\bar{\rho})}R_{\infty,\gamma}$ and 
$R^-_{\infty,\gamma} = 
R^{\square,\psi}(w_{0},\t_{\gamma}^-,\bar{\rho})
\otimes_{R^{\square,\psi}(w_{0},\t_{\gamma}^{ds},\bar{\rho})}R_{\infty,\gamma}$.
Then for all $\gamma$ 
the action of $R_{\infty,\gamma}$ on $\bar{M}_{\infty,\gamma}^{\pm}$ factors
through $R_{\infty,\gamma}^{\pm}$ as before.
Note that $\t_{\gamma}^- = \t_{\xi\gamma}^+$ and $R_{\infty,\gamma}^- = R_{\infty,\xi\gamma}^+$.

Note that the decompositions 
$M_{\infty} = M_{\infty}^+ \oplus M_{\infty}^-$ and
$\bar{M}_{\infty,\gamma} = \bar{M}_{\infty,\gamma}^+ \oplus \bar{M}_{\infty,\gamma}^-$
for all $\gamma$ are compatible, that is
$\bar{M}_{\infty}^{\pm}\otimes\F = \oplus_{\gamma}(\bar{M}_{\infty,\gamma}^{\pm})^{n_{\gamma}}$.

In particular, we have 
$e(M_{\infty}^{\pm}/\pi,R_{\infty}^{\pm}/\pi) = 
e(M_{\infty}^{\pm}/\pi ,R_{\infty}/\pi)$,
hence 
$e(M_{\infty}^{+}/\pi,R_{\infty}^{+}/\pi) + 
e(M_{\infty}^{-}/\pi,R_{\infty}^{-}/\pi) = 
e(M_{\infty}/\pi,R_{\infty}/\pi)$.
We also have that $e(M_{\infty}^{\pm}/\pi,R_{\infty}^{\pm}/\pi)
\leq e(R_{\infty}^{\pm}/\pi)$ by the same argument as in Lemma (2.2.11)
of \cite{Kis09a}.
Moreover
$d_{\t}e(R_{\infty}/\pi) = 
e(R^+_{\infty}/\pi) +
e(R^-_{\infty}/\pi)$ (see Proposition \ref{summult}).
Hence we deduce from Proposition \ref{equality} that
$e(M_{\infty}^{\pm}/\pi,R_{\infty}^{\pm}/\pi) = e(R_{\infty}^{\pm}/\pi)$
and
$e(\bar{M}_{\infty,\gamma}^{\pm},R_{\infty,\gamma}^{\pm}/\pi) 
= e(R_{\infty,\gamma}^{\pm}/\pi)$ for all $\gamma$.

Finally we get that
$e(R^{\square,\psi}(w,\t^{+},\bar{\rho})/\pi) = 
\sum_{\gamma}
[\bar{\sigmaG}(\t)\otimes\bar{\sigma_w}:\gamma]
e(R^{\square,\psi}(w_0,\t_{\gamma}^{+},\bar{\rho})/\pi)$, as
$e(R_{\infty}/\pi) \neq 0$.
As the right-hand side is 
$\mu_{\bar{\rho}}([\bar{\sigmaG}(\t)\otimes\bar{\sigma_w}])$
by the definition of $\mu_{\bar\rho}$ in Paragraph \ref{defmu}, we get
that
$e(R^{\square,\psi}(w,\t^+,\bar{\rho})/\pi) =
\mu_{\bar{\rho}}([\bar{\sigmaG}(\t)\otimes\bar{\sigma_w}])$
which finishes the proof of Theorems \ref{mainQp} and \ref{mainK}. 

\section{Application}
\label{appli}

\subsection{Computation of weights}

In this section we suppose $p \geq 5$.

Let $\bar{\rho}$ be a continuous representation $G_{\Q_p} \to \GL_2(\F)$
such that $\bar{\rho}|_{I_p} = \mat \omega{*}01\otimes\omega^n$.
We compute $\mu_{\bar{\rho}}$ in this case for the choice $\varpi_{\Q_p} = p$.

For a representation of the form
$\bar{\rho} = \mat{\omega}*01\otimes\omega^n\otimes\unr(x)$
then $\bar{\psi}(p) = x^2$. In order to apply Theorem \ref{mainQp} we
need to make a choice of a square root of $\bar{\psi}(p)^{-1}$, and 
we take this square root to be $\alpha = x^{-1}$.

\begin{lemm}
\label{multQp}
Let $\bar{\rho}|_{I_p} = \mat \omega{*}01\otimes \omega^n$ for some $n$
with $*$ très ramifié (and non-zero), then $\mu_{\bar{\rho}}(\xi^n\delta_n)= 1$ and all
other $\mu_{\bar{\rho}}(\gamma)$ are $0$. 

Let $\bar{\rho}|_{I_p} = \mat \omega{*}01\otimes \omega^n$ for some $n$
with $*$ peu ramifié but non-zero, then $\mu_{\bar{\rho}}(\xi^n\delta_n) = 1$ 
and $\mu_{\bar{\rho}}(r_{n(p+1)+p-1}) = 2$ and
all other $\mu_{\bar{\rho}}(\gamma)$ are $0$. 

Let $\bar{\rho}|_{I_p} = \mat \omega 001\otimes \omega^n$ for some $n$,
then $\mu_{\bar{\rho}}(\xi^n\delta_n) = 1$
and $\mu_{\bar{\rho}}(r_{n(p+1)+p-1}) = 4$ and
all other $\mu_{\bar{\rho}}(\gamma)$ are $0$.
\end{lemm}

\begin{proof}
We can compute 
$e(R^{\square,\psi}(w,\triv^{ds},\bar{\rho})/\pi)$ for any Hodge-Tate
type $w$ using the formula coming from the Breuil-Mézard conjecture for
$\GL_2$ and the list of the Serre weights with their multiplicities
given in \cite[Paragraph 2.1.2]{BM}.
Then we compare this to the formula for this multiplicity given by
Theorem \ref{mainQp}, using also the formula given by 
Lemma \ref{decsym}. 

We compute first
$e(R^{\square,\psi}(w,\triv^{ds},\bar{\rho})/\pi)$ for Hodge-Tate types
of the form $w = (0,m)$. We get that
$\mu_{\bar\rho}(\xi^{\alpha}\delta_m) = 0$ for $\alpha = 0,1$ if $m$ is
not equal to $n$ modulo $p-1$, and 
$\mu_{\bar\rho}(\delta_n) +\mu_{\bar\rho}(\xi\delta_n) = 1$. Using the
computations in \cite{BM}, Paragraph 5.2.1 we see that in fact
$\mu_{\bar{\rho}}(\xi^n\delta_n)= 1$.

By computing $e(R^{\square,\psi}(w,\triv^{ds},\bar{\rho})/\pi)$
for Hodge-Tate types $w$ of the form $(a,b)$ for $a > 0$ we can find the
value of $\mu_{\bar\rho}(r)$ for representations $r$ of dimension $2$.
\end{proof}

\subsection{An application to congruences modulo $p$ in
$S_k(\Gamma_0(p))$}

Let $f$ be a newform in $S_k(\Gamma_0(p))$. Then $a_p(f) = \pm
p^{k/2-1}$. We denote by $\rho_f$ the $p$-adic Galois representation
associated to $f$, $r_f$ its reduction modulo $p$,
and $r_{f,p}$ its restriction to a decomposition group
$G_{\Q_p}$ at $p$. 

\begin{theo}
\label{congruences}
Let $k > 2$ be an even integer, $f$ a newform in $S_k(\Gamma_0(p))$.
\begin{enumerate}
\item
Suppose that 
$r_{f,p}$ is of the form $\mat \omega{*}01\otimes
\omega^{k/2-1}\otimes\unr(x)$ for some $x$ and $*$ très ramifié
(and non-zero)
and $k \leq 2p+2$.
Then $x^{-1} = (-1)^{k/2-1}\bar{(a_p(f)/p^{k/2-1})}$. 
In particular, there does not exist a newform
$g$ in $S_k(\Gamma_0(p))$ congruent to $f$ modulo $p$
such that $a_p(g) = -a_p(f)$.
 
\item
Suppose that 
either $r_{f,p}|_{I_p}$ is not of the form $\mat \omega{*}01\otimes
\omega^{k/2-1}$ with $*$ très ramifié or that $k > 2p+2$.
If $r_f|_{G_{\Q(\zeta_p)}}$ is absolutely irreducible then there exists a newform
$g\in S_k(\Gamma_0(p))$ congruent to $f$ modulo $p$ such that $a_p(g) = -a_p(f)$.
\end{enumerate}
\end{theo}

\begin{proof}
Let $u_p(f) = a_p(f)p^{1-k/2}$. By the results of \cite{Sai}, $\rho_f$ is
a semi-stable, non-crystalline representation with extended type 
$\t_f =
(\|\cdot\|^{k/2-1}\oplus\|\cdot\|^{k/2})\otimes\unr(u_p(f)^{-1}) =
(1\oplus\|\cdot\|)\otimes\unr(u_p(f)^{-1}p^{1-k/2})$. 
Let $w_k$ be the Hodge-Tate type $(k-2,0)$ and $\psi_k = \eps^{k-2}$.

Suppose first that 
$r_{f,p}$ is of the form $\mat \omega{*}01\otimes
\omega^{k/2-1}\otimes\unr(x)$ with $*$ très ramifié.
By the existence of $f$,
$R^{\square,\psi_k}(w_k,\t_f,r_{f,p})$ is non-zero. 
With the normalization $\varpi_{\Q_p} = p$ as before, there is a choice of
$i\in\Z/2\Z$ with $\sigmaG(\triv) = \xi^i$ such that $\t_f =
(1\oplus\|\cdot\|)\otimes\unr((-1)^i)\otimes\unr(y^{-1}p^{1-k/2})$
for some $y$ lifting $x^{-1}$, and then
$e(R^{\square,\psi_k}(w_k,\t_f,r_{f,p})/\pi) =
\mu_{r_{f,p}}([\xi^i\sigma_{w_k}])$. As $k \leq 2p+2$, by Lemma
\ref{decsym} and Lemma \ref{multQp} this can be non-zero only if
$i=k/2-1$, that is $y = (-1)^{k/2-1}u_p(f)$,
which gives the result (note that we could apply the same method for $f
\in S_k(\Gamma_0(Np))$ new at $p$ for any $N$ such that $p \nmid N$).

Suppose now that either $r_{f,p}|_{I_p}$ is not of the form $\mat
\omega{*}01\otimes \omega^{k/2-1}$ with $*$ très ramifié or that $k > 2p+2$.
By the existence of $f$,
$R^{\square,\psi_k}(w_k,\t_f,r_{f,p})$ is non-zero
and then
Corollary \ref{samemult} or the computations of Lemma \ref{multQp}
show that both
$R^{\square,\psi_k}(w_k,\t_f,r_{f,p})$ and
$R^{\square,\psi_k}(w_k,\t_f',r_{f,p})$ are non-zero when $k>2$,
where $\t_f'$ is the extended type conjugate to $\t_f$.

Suppose now that moreover $r_f|_{G_{\Q(\zeta_p)}}$ is absolutely
irreducible.
Let $B$ be the quaternion algebra over $\Q$ that is ramified exactly
at $p$ and $\infty$. There exists
a modular form $f'$ on $B$
such that the automorphic representations attached to $f$ and $f'$
correspond to each other via Jacquet-Langlands,
and more precisely we can take for $f'$
an eigenform in
$S_{\sigma,\eta}(U,\O)$ for $\sigma = \sigma_{alg} = \Sym^{k-2}\O^2$ and
some character $\eta$ that restricts to $\psi_k$ at $p$,
and $U$ as in Paragraph \ref{residualrepr}. Then we are
in the situation of Paragraph \ref{patching}, 
from which we retain the notations.
In particular
Proposition \ref{equality} holds, hence the module $M_0[1/p]$
meets each
irreducible component of $\spec R_{p}[1/p]$. As $\spec R_p[1/p]$ has
irreducible components of both possible extended types, 
there exists an eigenform $g'$ in $S_{\sigma,\eta}(U,\O)$ such that the
Galois representation attached to $g'$ has an extended type at $p$ which
is conjugate to that of $r_f$. Let $g\in S_k(\Gamma_0(p))$ be an
eigenform such that the automorphic representations attached to $g$ and
$g'$ correspond via Jacquet-Langlands, then $g$ is the form we were
looking for.
\end{proof}

The first part of Theorem \ref{congruences} can be seen as a
generalization of Conjecture $4$ of \cite{CS} which was proved in
\cite{AB} (see also \cite{BP}).

\end{document}